\documentclass[10pt,a4paper]{amsart}
\usepackage{amsfonts}
\usepackage{amsthm}
\usepackage{setspace}
\usepackage{amsmath}
\usepackage{amscd}
\usepackage[latin2]{inputenc}
\usepackage{t1enc}
\usepackage[mathscr]{eucal}
\usepackage{indentfirst}
\usepackage{graphicx}
\usepackage{graphics}
\usepackage{pict2e}
\usepackage{epic}
\numberwithin{equation}{section}

\usepackage{epstopdf} 
\usepackage{amssymb}
\usepackage{enumerate}
\usepackage[all]{xy}
\usepackage{graphicx}%
\usepackage{tikz-cd}
\usepackage{mathtools}

\theoremstyle{plain}
\newtheorem{Th}{Theorem}[section]
\newtheorem{theorem}[Th]{Theorem}
\newtheorem{lemma}[Th]{Lemma}
\newtheorem{corollary}[Th]{Corollary}
\newtheorem{proposition}[Th]{Proposition}

 \theoremstyle{definition}
\newtheorem{definition}[Th]{Definition}

\newtheorem{?}[Th]{Problem}

\newcommand{\im}{\operatorname{im}}
\newcommand{\coker}{\operatorname{coker}}
\newcommand{\Hom}{{\rm{Hom}}}

\DeclareMathOperator{\Tor}{Tor}

\newcommand{\Mod}[1]{\ (\mathrm{mod}\ #1)}

\begin{document}

\title{The cyclic homology of $k[x_1,x_2,\ldots,x_d]/(x_1,x_2,\ldots,x_d)^2$}
\author{Emily Rudman}
\address{Mathematics Department, The Ohio State University, 100 Math Tower, 231 W 18th Ave.,
  Columbus, OH 43210, USA}  
\email{rudman.21@osu.edu}

\date{\today}

\maketitle



\begin{abstract}
 The Hochschild homology of the ring $k[x_1,x_2,\ldots,x_d]/(x_1,x_2,\ldots,x_d)^2$ has been known and calculated several ways. This paper uses those calculations to calculate cyclic, negative cyclic, and periodic cyclic homology of $k[x_1,x_2,\ldots,x_d]/(x_1,x_2,\ldots,x_d)^2$ over $k$.
\end{abstract}

\maketitle


\section{Introduction}

    The goal of this paper is to compute the cyclic homology, negative cyclic homology, and periodic cyclic homology of the ring $A =k[x_1,x_2,\ldots,x_d]/(x_1,x_2,\ldots,x_d)^2$ over $k$ for $k$ an arbitrary commutative ring and to give explicit formulas for the cases $k = \mathbb{Q}$ and $k = \mathbb{Z}$. Since the complexes calculating all these homologies for $k[x_1,x_2,\ldots,x_d]/(x_1,x_2,\ldots,x_d)^2$ over $k$ can be obtained from the complexes calculating these homologies for $\mathbb{Z}[x_1,x_2,\ldots,x_d]/(x_1,x_2,\ldots,x_d)^2$ over $\mathbb{Z}$, the explicit results for general $k$ can be obtained from those for $\mathbb{Z}$ by using the Universal Coefficient Theorem.

An important motivation for studying Hochschild-type invariants is given by the Dennis trace map from algebraic $K$-theory to Hochschild homology. It factors through negative cyclic homology, 

\[\begin{tikzcd}[column sep=small]
K_{*}(R) \arrow{dr}{} \arrow{rr}{\text{Dennis trace}}& & HH_{*}(R) \\
 & HC_{*}^{-}(R) \arrow{ur}{} 
\end{tikzcd}
\]

$\\*$and it has been proven by Goodwillie \cite{GW} that $HC_{*}^{-}(R)$ is a much better approximation of $K_{*}(R)$ than $HH_{*}(R)$. 

Section 2 introduces the needed definitions and Section 3 contains the calculations. It starts by showing how the Hochschild complex breaks down by weights and cyclic words. We extend this idea to the Tsygan double complex and then compare the result to a specific Tor computation to get the general cyclic homology result. 

Section 4 computes the cyclic homology for $k=\mathbb{Q}$, which follows easily from Section 3 since $\mathbb{Q}$ is a projective $\mathbb{Q}[C_{w}]$-module. Section 5 requires more careful analysis for $k=\mathbb{Z}$. As explained above, for a general ring $k$, since
$$
\big(k[x_1,x_2,\ldots,x_d]/(x_1,x_2,\ldots,x_d)^2\big)^{\otimes_{k} \ell} \cong k\otimes_{\mathbb{Z}}\big(\mathbb{Z}[x_1,x_2,\ldots,x_d]/(x_1,x_2,\ldots,x_d)^2\big)^{\otimes_{\mathbb{Z}} \ell}
$$
by the Universal Coefficient Theorem \newline$
HH_{\ell}^{k}\big(k[x_1,x_2,\ldots,x_d]/(x_1,x_2,\ldots,x_d)^2\big) \cong 
$
$$
k\otimes_{\mathbb{Z}} HH_{\ell}^{\mathbb{Z}}\big(\mathbb{Z}[x_1,\ldots,x_d]/(x_1,\ldots,x_d)^2\big) \bigoplus \Tor \big(k, HH_{\ell - 1}^{\mathbb{Z}}\big(\mathbb{Z}[x_1,\ldots,x_d]/(x_1,\ldots,x_d)^2\big) \big),
$$
$\\*$and similarly for cyclic, negative cyclic, and periodic cyclic homology.

After the cyclic homology computations, in Section 6 and Section 7 the negative cyclic homology and periodic cyclic homology computations are similarly given for  $k$, $k=\mathbb{Q}$, and $k=\mathbb{Z}$.

\section{Definitions}
\begin{definition}\label{HB}
Let $k$ be a commutative ring with unit and $A$ be a $k$-algebra with unit. The \textbf{Hochschild complex} of $A$ over $k$ consists in degree $n$ of 
$$
C_{n}(A) = A^{\otimes n+1} \coloneqq \underbrace{A \otimes_{k} A \otimes_{k} \cdots \otimes_{k} A \otimes_{k} A }_{n+1 \textrm{ times}}$$.
$\\*$with respect to the Hochschild boundary
\begin{center}
      $b_{n}(a_{0}\otimes a_{1}\otimes\cdots\otimes a_{n}) =  \sum\limits_{i=0}^{n} (-1)^{i} d_{i}(a_{0}\otimes a_{1}\otimes\cdots\otimes a_{n}),$
\end{center}
where $d_{i}: A^{\otimes n+1} \rightarrow A^{\otimes n}$, $0 \leq i \leq n$ are defined as
\begin{center}
 $d_{i}(a_{0}\otimes\cdots\otimes a_{n}) = a_{0}\otimes \cdots\otimes a_{i}a_{i+1}\otimes\cdots\otimes a_{n}$
 
 $d_{n}(a_{0}\otimes a_{1}\otimes\cdots\otimes a_{n}) = a_{n}a_{0}\otimes a_{1} \otimes \cdots\otimes a_{n-1}$.
 \end{center}
 \end{definition}
\begin{definition}We also define $b'_{n}:A^{\otimes n+1} \rightarrow A^{\otimes n}$ by

\begin{center}
    
$b'_{n}(a_{0}\otimes a_{1}\otimes\cdots\otimes a_{n}) =  \sum\limits_{i=0}^{n-1} (-1)^{i} d_{i}(a_{0}\otimes a_{1}\otimes\cdots\otimes a_{n}).$

\end{center}
The following chain complex $C'_{*}(A)$ is called the $\textbf{Bar}$ $\textbf{Resolution}$ $\textbf{of}$ $\textbf{A}$:

\medskip

\begin{center}
    
$C'_{*}(A) = \cdots\xrightarrow{b'_{n+1}} A^{\otimes n+1} \xrightarrow{\mathrel{\phantom{=}}b'_{n }\mathrel{\phantom{=}}} A^{\otimes n} \xrightarrow {b'_{n-1}} A^{\otimes n-1} \xrightarrow {b'_{n-2}}\cdots \xrightarrow{\mathrel{\phantom{=}}b'_{2}\mathrel{\phantom{=}}} A^{\otimes 2} \xrightarrow{\mathrel{\phantom{=}}b'_{1}\mathrel{\phantom{=}}} A \xrightarrow{\mathrel{\phantom{=}}\mathrel{\phantom{=}}\mathrel{\phantom{=}}} 0 $ .

\end{center}
In particular, it is acyclic.
\end{definition}

\begin{definition}\label{NHC}
 The \textbf{normalized or reduced Hochschild complex}, denoted by $\bar{C_{*}}({A})$, is the quotient of $C_{*}({A})$ by all sums of degenerate elements (elements $a_{0}\otimes a_{1}\otimes\cdots\otimes a_{n}$ where at least one of $a_{1}, \hdots, a_{n}$ is in $k$). The reduction map is a quasi-isomorphism from the Hochschild complex to the reduced Hochschild complex by Lemma 1.1.15 \cite{CH} with the maps $b_{n}$ induced by those of the complex $C_{*}(A)$. $\bar{C_{*}}({A})$ is written as

\begin{center}
$\bar{C_{*}}({A})=\cdots\xrightarrow{b_{n+1}} A\otimes \bar{A}^{\otimes n} \xrightarrow{\mathrel{\phantom{=}}b_{n}\mathrel{\phantom{=}}} A\otimes\bar{A}^{\otimes n-1} \xrightarrow {b_{n-1}}\cdots \xrightarrow{\mathrel{\phantom{=}}b_{2}\mathrel{\phantom{=}}} A\otimes\bar{A} \xrightarrow{\mathrel{\phantom{=}}b_{1}\mathrel{\phantom{=}}} A \xrightarrow{\mathrel{\phantom{=}}\mathrel{\phantom{=}}\mathrel{\phantom{=}}}0$.

\end{center}

 \end{definition}

\begin{definition}\label{CH} The \textbf{cyclic homology of A}, denoted by $HC_{*}(A)$, is the total homology of the first quadrant double complex:

\begin{center}
\begin{tikzcd}
    \vdots\arrow{d} & \vdots\arrow{d} & \vdots\arrow{d} & \vdots\arrow{d} \\
    A^{\otimes 3}\arrow{d}{b_{2}} & A^{\otimes 3}\arrow{l}{1-t_{2}}\arrow{d}{b'_{2}} & A^{\otimes 3}\arrow{l}{N_{2}}\arrow{d}{b_{2}} & A^{\otimes 3}\arrow{l}{1-t_{2}}\arrow{d}{b'_{2}} &{\cdots} \arrow{l}{N_{2}} \\
    A^{\otimes 2}\arrow{d}{b_{1}} & A^{\otimes 2}\arrow{l}{1-t_{1}}\arrow{d}{b'_{1}} & A^{\otimes 2}\arrow{l}{N_{1}}\arrow{d}{b_{1}} & A^{\otimes 2}\arrow{l}{1-t_{1}}\arrow{d}{b'_{1}} &{\cdots}\arrow{l}{N_{1}} \\
    A & A\arrow{l}{1-t_{0}} & A \arrow{l}{N_{0}} & A\arrow{l}{1-t_{0}} &{\cdots}\arrow{l}{N_{0}}
\end{tikzcd}
\end{center}
where the maps $t_{n}$ and $N_{n}$ are defined as
\begin{equation}\label{tn}
t_{n}(a_{0}\otimes a_{1}\otimes \cdots\otimes a_{n}) = (-1)^{n}(a_{n}\otimes a_{0}\otimes a_{1} \otimes \cdots \otimes a_{n-1})
\end{equation}
\begin{equation}\label{Nn}
N_{n} = 1 + t_{n} + t_{n}^{2} + \cdots + t_{n}^{n}.
\end{equation}
\end{definition}
This first quadrant double complex is called the Tsygan complex or the cyclic bicomplex. We can compute the total homology of the Tsygan complex using the spectral sequence associated to filtration by columns.  The $E^{2}$ page consists of $HH_{*}(A)$ in even columns and $0$ in odd ones.

\begin{center}
\begin{tikzcd}
    \vdots & \vdots& \vdots & \vdots& \vdots\\
    HH_{2}(A) &0 & HH_{2}(A) & 0 & HH_{2}(A)&{\cdots}  \\
   HH_{1}(A) & 0 & HH_{1}(A)\arrow{llu}{\partial^{2}}& 0 & HH_{1}(A)\arrow{llu}{\partial^{2}}&{\cdots} \\
    HH_{0}(A) & 0 & HH_{0}(A)\arrow{llu}{\partial^{2}} & 0 & HH_{0}(A)\arrow{llu}{\partial^{2}}&{\cdots}
\end{tikzcd}
\end{center}



$\\*$The map $s:A^{\otimes n} \rightarrow A^{\otimes n+1}$, given by 
$s(a_{0}\otimes a_{1} \otimes \cdots \otimes a_{n}) = 1\otimes a_{0}\otimes a_{1} \otimes \cdots \otimes a_{n}$,
 is a lift of the map $b'_{n}$ on $\ker(b'_{n-1})$. In other words, it is a chain homotopy between the identity map and the zero map on $(A^{\otimes n+1}, b')$. So $\partial^{2}: HH_{n-1}(A) \rightarrow HH_{n}(A)$ can be induced on the chain level by $(1-t_{n}) s  N_{n-1}$. This leads to another bicomplex that can be used to compute $HC_{*}(A)$.

\begin{definition} The \textbf{Connes boundary map $B$} is the map $B: A^{\otimes n} \xrightarrow{} A^{\otimes (n+1)}$ given by:
\begin{center}
  $B= (1-t_{n}) s  N_{n-1}$.
\end{center}
One can equivalently look at the bicomplex $B(A)$
\begin{center}
\begin{tikzcd}
    \vdots\arrow{d} & \vdots\arrow{d} & \vdots\arrow{d}\\
    A^{\otimes 3}\arrow{d}{b_{2}} & A^{\otimes 2}\arrow{d}{b_{1}}\arrow{l}{B} &A\arrow{l}{B} \\
    A^{\otimes 2}\arrow{d}{b_{1}} &  A\arrow{l}{B} &\\
    A &  &  
\end{tikzcd}
\end{center}
which essentially eliminates the odd-numbered columns. Once the even-numbered columns are shifted to the left, they need to be raised to preserve the total degree. By \cite{CH} (Theorem 2.1.8), the total homology of $B(A)$ is also $HC_{*}(A)$.
\end{definition}
\begin{definition}\label{negative} The \textbf{negative cyclic homology of A}, denoted by $HC^{-}_{*}(A)$,  is the total homology of the second quadrant double complex:

\begin{center}
\begin{tikzcd}
    & \vdots\arrow{d} & \vdots\arrow{d} & \vdots\arrow{d} & \vdots\arrow{d} \\
   {\cdots}&  A^{\otimes 3}\arrow{l}{1-t_{2}}\arrow{d}{b'_{2}} & A^{\otimes 3}\arrow{l}{N_{2}}\arrow{d}{b_{2}} & A^{\otimes 3}\arrow{l}{1-t_{2}}\arrow{d}{b'_{2}} & A^{\otimes 3}\arrow{l}{N_{2}}\arrow{d}{b_{2}} & \\
   {\cdots}&  A^{\otimes 2}\arrow{l}{1-t_{1}}\arrow{d}{b'_{1}} & A^{\otimes 2}\arrow{l}{N_{1}}\arrow{d}{b_{1}} & A^{\otimes 2}\arrow{l}{1-t_{1}}\arrow{d}{b'_{1}} & A^{\otimes 2}\arrow{l}{N_{1}}\arrow{d}{b_{1}} & \\
  {\cdots}& A\arrow{l}{1-t_{0}} & A\arrow{l}{N_{0}} & A \arrow{l}{1-t_{0}} & A\arrow{l}{N_{0}} &
\end{tikzcd}
\end{center}
\end{definition}

\begin{definition}\label{periodic} The \textbf{periodic cyclic homology of A}, denoted by $HP_{*}(A)$, is the total homology of the first and second quadrant double complex:

\begin{center}
\begin{tikzcd}
    & \vdots\arrow{d} & \vdots\arrow{d} & \vdots\arrow{d} &  \\
   {\cdots}&  A^{\otimes 3}\arrow{l}{1-t_{2}}\arrow{d}{b'_{2}} & A^{\otimes 3}\arrow{l}{N_{2}}\arrow{d}{b_{2}} & A^{\otimes 3}\arrow{l}{1-t_{2}}\arrow{d}{b'_{2}} & {\cdots}\arrow{l}{N_{2}} & \\
   {\cdots}&  A^{\otimes 2}\arrow{l}{1-t_{1}}\arrow{d}{b'_{1}} & A^{\otimes 2}\arrow{l}{N_{1}}\arrow{d}{b_{1}} & A^{\otimes 2}\arrow{l}{1-t_{1}}\arrow{d}{b'_{1}} & {\cdots}\arrow{l}{N_{1}} & \\
  {\cdots}& A\arrow{l}{1-t_{0}} & A\arrow{l}{N_{0}} & A \arrow{l}{1-t_{0}} & {\cdots}\arrow{l}{N_{0}} &
\end{tikzcd}
\end{center}
$\\*$where the even columns are the Hochschild complex and the odd columns are the bar complex. 
\end{definition}
Cyclic homology is closely connected to taking a homotopy quotient of a circle action on Hochschild homology, whereas negative cyclic homology is closely connected to taking homotopy fixed points of that action.

\section{Cyclic Homology of $k[x_1,x_2,\hdots,x_d]/\mathfrak{m}^2$}

\subsection{Grading HH by weight} \hfill

Let $A = k[x_1,x_2,\hdots,x_d]/\mathfrak{m}^2$, where $k$ is any commutative unital ring and $\mathfrak{m}$ is the ideal $(x_1,x_2,\hdots,x_d)$. Computing $HH_{n}(A)$ is easier to do if we break the Hochschild complex down by weights, where the weight  of a tensor monomial is the total number of $x'_{i}s$ in it. We can do this because the Hochschild boundary maps preserve weight. Thus,    $C_{*}(A) \cong \displaystyle\bigoplus_{w=0}^{\infty}C_{*}^{(w)}(A)$ where $C_{*}^{(w)}(A)$ is the subcomplex consisting of all elements in $C_{*}(A)$ of weight $w$ and
$HH_{*}(A) \cong \displaystyle\bigoplus_{w=0}^{\infty}HH_{*}^{(w)}(A).$

When computing $HH_{n}^{(w)}(A)$, we can again look at the normalized Hochschild complex. The only nondegenerate elements of weight $w \neq 0$ are in the $(w)^{th}$ and $(w-1)^{th}$ levels of the normalized complex. The only nondegenerate elements of weight $w=0$ are elements from $k$ in the $0^{th}$ level of the complex. The elements of weight $w$ (for $w > 0$) that are in the ($w$)$^{th}$ level of the normalized complex are spanned by the nondegenerate tensor monomials
\begin{center}
$ 1\otimes x_{j_{1}}\otimes\cdots\otimes x_{j_{w}}   $ where $j_{i} \in \{1,2,\hdots,d\}$.
\end{center}
The elements of weight $w$ (for $w > 0$) that are in the ($w-1$)$^{th}$ level of the reduced complex are spanned by the nondegenerate tensor monomials 
\begin{center}
$x_{j_{1}}\otimes\cdots\otimes x_{j_{w}}   $ where $j_{i} \in \{1,2,\hdots,d\}$.
\end{center}
Let $\bar{\mathfrak{m}} = \mathfrak{m} / \mathfrak{m}^{2}$. Note that $\bar{\mathfrak{m}}$ is the free $k$-module on $[x_{1}]$, $[x_{2}], \hdots,[x_{d}]$ which by abuse of notation we call $x_{1}$, $x_{2},\hdots,x_{d}$. Therefore, the normalized Hochschild complex for weight $w$ (for $w > 0$) is
\begin{center}
\begin{tikzcd}
    
   {\cdots}\arrow{r}&  0\arrow{r} & 1\otimes \bar{\mathfrak{m}}^{\otimes w}\arrow{r}{b_{w}} & \bar{\mathfrak{m}}^{\otimes w}\arrow{r}& 0\arrow{r} & {\cdots}
\end{tikzcd}.
\end{center}
Consider the following chain map:
\begin{center}
\begin{tikzcd}
    
   {\cdots}\arrow{r}&  0\arrow{r}\arrow{d} & 1\otimes \bar{\mathfrak{m}}^{\otimes w}\arrow{r}{b_{w}}\arrow{d}{f} & \bar{\mathfrak{m}}^{\otimes w}\arrow{r}\arrow{d}{id} & 0\arrow{r}\arrow{d} & {\cdots}\\
  {\cdots}\arrow{r}& 0\arrow{r} & \bar{\mathfrak{m}}^{\otimes w}\arrow{r}{1-t_{w-1}} & \bar{\mathfrak{m}}^{\otimes w} \arrow{r} & 0\arrow{r} & {\cdots}\> ,
\end{tikzcd}
\end{center}
where $f(1\otimes x_{j_{1}}\otimes\cdots\otimes x_{j_{w}})=x_{j_{1}}\otimes\cdots\otimes x_{j_{w}}$. Note
\begin{center}
$
b_{w}(1\otimes x_{j_{1}}\otimes\cdots\otimes x_{j_{w}})  = (1-t_{w-1})(x_{j_{1}}\otimes\cdots\otimes x_{j_{w}}).
$
\end{center}
This chain map is an isomorphism so it induces an isomorphism on homology, yielding
\begin{equation}\label{HHofRing1} HH_{n}^{(0)}(A) \cong 
    \begin{dcases}
        k  & n = 0 \\
        0 & \textrm{else}
    \end{dcases}
\end{equation}
 and for $w > 0$

\begin{equation}\label{HHofRing2} HH_{n}^{(w)}(A) \cong 
    \begin{dcases}
        \ker\big((1-t_{w-1}):\bar{\mathfrak{m}}^{\otimes w}\rightarrow \bar{\mathfrak{m}}^{\otimes w}\big)  & n=w \\
        \coker\big((1-t_{w-1}):\bar{\mathfrak{m}}^{\otimes w}\rightarrow \bar{\mathfrak{m}}^{\otimes w}\big) & n= w-1 \\
        0 & \textrm{else} \> . 
    \end{dcases}
\end{equation}
\subsection{Relating the reduced Hochschild complex to the Tsygan complex}\hfill

We now want to calculate $HC_{*}(A)$ for $A = k[x_1,x_2,\hdots,x_d]/\mathfrak{m}^2$. Since $b'_{n}$, $t_{n}$, and $N_{n}$ as well as $b_{n}$ preserve weight, the Tsygan complex and all pages of the spectral sequence and its homology split by weight as well, so $HC_{*}(A) \cong \displaystyle\bigoplus_{w=0}^{\infty}HC_{*}^{(w)}(A).$ The calculation of $HC_{*}^{(w)}(A)$  breaks down into two cases, one where the weight $w=0$ and the other where $w > 0$.
$\\*$\textbf{Case 1: $w = 0$}
$\\*$On the $E^{1}$-page, we have $E^{1}_{p,q}=0$ unless $q=0$ and $p\geq0$ is even, so $E^{1} \cong E^{\infty}$. Therefore

\begin{equation}\label{HCW0}
HC_{n}^{(0)}(A) \cong 
    \begin{dcases}
        k  & n \textrm{ even and } n \geq 0\\
        0 & \textrm{else}. \\
    \end{dcases}
\end{equation}
\textbf{Case 2: $w > 0$}
$\\*$ In the $E^{1}$-page, the even columns consist of $HH_{*}(A)$ and the odd columns are zero. Since the odd columns are zero, $\partial^{1}=0$ and $E^{1} \cong E^{2}$. By Equation (\ref{HHofRing2}), $E^{2}$ is given by:

\begin{equation}\label{double}
\begin{tikzcd}
   & \vdots & \vdots & \vdots & \vdots  &   \\
  (w+1)^{th} & 0 & 0 & 0 & 0  & \cdots \\
  (w)^{th} & \ker(1-t_{w-1}) & 0 & \ker(1-t_{w-1}) & 0  & \cdots \\
  (w-1)^{th} &  \textrm{coker}(1-t_{w-1}) & 0 & \textrm{coker}(1-t_{w-1})\arrow{ull}{\partial^{2}} & 0  &\arrow{ull}{\partial^{2}}  \cdots\\
   (w-2)^{nd} & 0 & 0 & 0 & 0  & \cdots \\
     & 0^{th} &1^{st}   &2^{nd}   &3^{rd}  & 
 \end{tikzcd}
\end{equation}
All $\partial^{r}$ for $r\geq 3$ are trivial for dimension reasons, so the $E^{3}$ page is same as the $E^{\infty}$ page. Therefore,
 \begin{equation}\label{HCNW}
HC_{n}^{(w)}(A) \cong 
    \begin{dcases}
        0  & n \leq w-2 \\
        \textrm{coker}\big((1-t_{w-1}):\bar{\mathfrak{m}}^{\otimes w}\rightarrow \bar{\mathfrak{m}}^{\otimes w}\big) & n = w-1 \\
        \textrm{coker}(\partial^{2}) & n = w + 2i,\hspace{11.5mm} i\geq 0\\
        \ker(\partial^{2}) & n = w + 1 + 2i, \hspace{5mm} i\geq 0\> . \\
    \end{dcases}
\end{equation}
As discussed above, $\partial^{2}$ is the map induced on the homology classes by $B = (1-t_{w})sN_{w-1}$.
\begin{center}
\begin{tikzcd}
 & (w)^{th}  & A^{\otimes w+1}  & A^{\otimes w+1} \arrow{l}{1-t_{w}} & \\
 & (w-1)^{th} &  & A^{\otimes w} \arrow{u}{s} & A^{\otimes w} \arrow{l}{N_{w-1}} \> \\
 \end{tikzcd}
\end{center}
$HH_{w-1}^{(w)}(A)$ is spanned by the homology classes of all cycles $x_{k_{1}}\otimes \cdots \otimes x_{k_{w}}$ for $k_{1},\hdots, k_{w} \in \{1, \hdots , d\}$, and on such a cycle, $(1-t_{w-1})sN_{w-1}(x_{k_{1}}\otimes \cdots \otimes x_{k_{w}})=sN_{w-1}(x_{k_{1}}\otimes \cdots \otimes x_{k_{w}})+$ degenerate elements. So if we reduce the range $A^{\otimes w+1} \rightarrow A\otimes \bar{A}^{\otimes w}$, we can regard the map as $sN_{w-1}$. If we further think of the reduced Hochschild complex of weight $w$ as  $\bar{\mathfrak{m}}^{\otimes w}\xrightarrow{1-t_{w-1}}\bar{\mathfrak{m}}^{\otimes w}$ as described in the previous section,  $\partial^{2}$ can be viewed as
\begin{center}
$N_{w-1}:\textrm{coker}(1-t_{w-1}) \xrightarrow{} \ker(1-t_{w-1})$
\end{center}
and we can rewrite Equation (\ref{HCNW}) as 
\begin{equation}\label{HCNW2}
    HC_{n}^{(w)}(A) \cong 
    \begin{dcases}
        0  & n \leq w-2 \\
        \textrm{coker}(1-t_{w-1}) & n = w-1 \\
        \textrm{coker}\big(\textrm{coker}(1-t_{w-1}) \xrightarrow{N_{w-1}} \ker(1-t_{w-1})\big) & n = w + 2i, \hspace{11.5mm} i\geq 0  \\
        \ker\big(\textrm{coker}(1-t_{w-1}) \xrightarrow{N_{w-1}} \ker(1-t_{w-1})\big) & n = w + 1 + 2i, \hspace{5mm} i\geq 0  \> .\\
    \end{dcases}
\end{equation}
\subsection{Comparison to Tor}\hfill

The calculation of the pieces of Equation (\ref{HCNW2}) is exactly what we get computing $Tor_{n}^{k[C_{w}]}(k,\bar{\mathfrak{m}}^{\otimes w})$ where $C_{w}= \langle \alpha : \alpha^{(w)}=1 \rangle$ is the cyclic group of order $w$ and $\bar{\mathfrak{m}}^{\otimes w}$ is a $k[C_{w}]$-module with the action $\alpha(x_{n_{1}}\otimes\cdots\otimes x_{n_{w}}) = t_{w-1}(x_{n_{1}}\otimes\cdots\otimes x_{n_{w}}) = (-1)^{w-1}(x_{n_{w}}\otimes x_{n_{1}} \otimes\cdots\otimes x_{n_{w-1}}).
$ For any commutative ring $k$, the following is a projective $k[C_{w}]$-resolution of $k$:
\begin{equation}\label{r}
  {\cdots}\xrightarrow{1+\alpha + {\cdots} +\alpha^{w-1}}k[C_{w}]\xrightarrow{\mathrel{\phantom{=}}1-\alpha\mathrel{\phantom{=}}}k[C_{w}]\xrightarrow{1+\alpha + {\cdots} + \alpha^{w-1}}k[C_{w}]\xrightarrow{\mathrel{\phantom{=}}1-\alpha\mathrel{\phantom{=}}}k[C_{w}]\xrightarrow{\mathrel{\phantom{=}}r\mathrel{\phantom{=}}} k\xrightarrow{}0 \> . 
\end{equation}
Here, $C_{w}= \langle \alpha : \alpha^{(w)}=1  \rangle$ and $k[C_{w}]\xrightarrow{\mathrel{\phantom{=}}r\mathrel{\phantom{=}}} k$ is the augmentation $r\Bigg(\sum\limits_{j=0}^{w-1} k_{j}\cdot  \alpha^{j} \Bigg) =\sum\limits_{j=0}^{w-1}k_{j}  $ .

After deleting $k$, tensoring over $k[C_{w}]$ with $\bar{\mathfrak{m}}^{\otimes w}$, and setting $N = (1+\alpha + {\cdots} + \alpha^{w-1})$ we get
\begin{equation}\label{complexmbar}
\cdots\xrightarrow{1-t_{w-1}} \bar{\mathfrak{m}}^{\otimes w}\xrightarrow{N_{w-1}} \bar{\mathfrak{m}}^{\otimes w}\xrightarrow{1-t_{w-1}} \bar{\mathfrak{m}}^{\otimes w}\xrightarrow{N_{w-1}} \bar{\mathfrak{m}}^{\otimes w}\xrightarrow{1-t_{w-1}}\bar{\mathfrak{m}}^{\otimes w}\xrightarrow{} 0
\end{equation}
yielding
\begin{equation}\label{tor0}
Tor_{0}^{k[C_{w}]}(k,\bar{\mathfrak{m}}^{\otimes w})\cong  \textrm{coker}(1-t_{w-1})
\end{equation}
\begin{equation}\label{tor1}
Tor_{1 + 2i}^{k[C_{w}]}(k,\bar{\mathfrak{m}}^{\otimes w})\cong\textrm{coker}\big(\textrm{coker}(1-t_{w-1}) \xrightarrow{N_{w-1}} \ker(1-t_{w-1})\big) 
\end{equation}
\begin{equation}\label{tor2}
Tor_{2+ 2i}^{k[C_{w}]}(k,\bar{\mathfrak{m}}^{\otimes w})\cong\ker\big(\textrm{coker}(1-t_{w-1}) \xrightarrow{N_{w-1}} \ker(1-t_{w-1})\big) .
\end{equation}
So we can rewrite Equation (\ref{HCNW2}) as
\begin{theorem}\label{HCNW3}
Let $A = k[x_1,x_2,\hdots,x_d]/\mathfrak{m}^2$, where $k$ is any commutative unital ring, $\mathfrak{m}$ is the ideal $(x_1,x_2,\hdots,x_d)$, $\bar{\mathfrak{m}} = \mathfrak{m} / \mathfrak{m}^{2}$, $C_{w}= \langle \alpha : \alpha^{(w)}=1  \rangle$, and $\bar{\mathfrak{m}}^{\otimes w}$ is a $k[C_{w}]$-module with the action:
$$
\alpha(x_{n_{1}}\otimes\cdots\otimes x_{n_{w}}) = (-1)^{w-1}(x_{n_{w}}\otimes x_{n_{1}} \otimes\cdots\otimes x_{n_{w-1}}) \> .
$$
Then for $w = 0$
\begin{center}
$
HC_{n}^{(0)}(A) \cong 
    \begin{dcases}
        k  & n \textrm{ even and } n \geq 0 \\
        0 & \textrm{else} \\
    \end{dcases}
$
\end{center}
and for $w > 0$
\begin{center}
$
HC_{n}^{(w)}(A) \cong 
    \begin{dcases} 
        0  & n \leq w-2 \\
       Tor_{n-w+1}^{k[C_{w}]}(k,\bar{\mathfrak{m}}^{\otimes w}) & n \geq w-1\>  \\
    \end{dcases}
$
\end{center}

$\\*$ which can be rewritten as 

\begin{center}
$
HC_{n}^{(w)}(A) \cong 
    \begin{dcases} 
        0  & n \leq w-2 \\
       H_{n-w+1}(C_{w};\bar{\mathfrak{m}}^{\otimes w}) & n \geq w-1\>.  \\
    \end{dcases}
$
\end{center}
\end{theorem}
\section{Cyclic Homology of $\mathbb{Q}[x_1,x_2,\hdots,x_d]/\mathfrak{m}^2$}

Since $\mathbb{Q}$ is projective over $\mathbb{Q}[C_{w}]$, $Tor_{n}^{\mathbb{Q}[C_{w}]}(\mathbb{Q},\bar{\mathfrak{m}}^{\otimes w}) = 0$ for $n >0$. Therefore, we get that for $w = 0$

\begin{center}
$
HC_{n}^{(0)}(A) \cong 
    \begin{dcases}
        \mathbb{Q}  & n \textrm{ even and } n \geq 0  \\
        0 & \textrm{else} \\
    \end{dcases}
$
\end{center}
$\\*$ and for $w > 0$

\begin{center}
$
HC_{n}^{(w)}(A) \cong 
    \begin{dcases} 
       \bar{\mathfrak{m}}^{\otimes w} / (x_{n_{1}}\otimes {\cdots} \otimes x_{n_{w}} \sim (-1)^{w-1}x_{n_{w}}\otimes x_{n_{1}}\otimes {\cdots} \otimes x_{n_{w-1}}) & n = w-1 \\
      \quad \quad \quad \forall n_{1}, \hdots, n_{w} \in \{1,\hdots, d\} \\
        0  & \textrm{else}.\> 
    \end{dcases}
$
\end{center}

We later define the cycle length of a tensor monomial to be the smallest $m>0$ such that if you rotate the last $m$ coordinates of the monomial to the front, your monomial doesn't change. We can rewrite the tensor monomial as $x_{n_{1}}\otimes {\cdots} \otimes x_{n_{w}} = (x_{k_{1}}\otimes\cdots\otimes x_{k_{m}})^{\otimes \ell}$. Notice that when $w$ is odd, $(x_{n_{1}}\otimes {\cdots} \otimes x_{n_{w}} \sim (-1)^{w-1}x_{n_{w}}\otimes x_{n_{1}}\otimes {\cdots} \otimes x_{n_{w-1}})$ gives $$
(x_{k_{1}}\otimes\cdots\otimes x_{k_{m}})^{\otimes \ell}
\sim
(x_{k_{m}}\otimes\cdots\otimes x_{k_{m-1}})^{\otimes \ell}
\sim
\cdots
\sim
(x_{k_{3}}\otimes\cdots\otimes x_{k_{2}})^{\otimes \ell}
\sim
(x_{k_{2}}\otimes\cdots\otimes x_{k_{1}})^{\otimes \ell},
$$
when $w$ is even and $m$ is even, $(x_{n_{1}}\otimes {\cdots} \otimes x_{n_{w}} \sim (-1)^{w-1}x_{n_{w}}\otimes x_{n_{1}}\otimes {\cdots} \otimes x_{n_{w-1}})$ gives $$
(x_{k_{1}}\otimes\cdots\otimes x_{k_{m}})^{\otimes \ell}
\sim
-(x_{k_{m}}\otimes\cdots\otimes x_{k_{m-1}})^{\otimes \ell}
\sim
\cdots
\sim
-(x_{k_{2}}\otimes\cdots\otimes x_{k_{1}})^{\otimes \ell}
\sim
(x_{k_{1}}\otimes\cdots\otimes x_{k_{m}})^{\otimes \ell},
$$
but when $w$ is even and $m$ is odd, $(x_{n_{1}}\otimes {\cdots} \otimes x_{n_{w}} \sim (-1)^{w-1}x_{n_{w}}\otimes x_{n_{1}}\otimes {\cdots} \otimes x_{n_{w-1}})$ gives
$$
(x_{k_{1}}\otimes\cdots\otimes x_{k_{m}})^{\otimes \ell}
\sim
-(x_{k_{m}}\otimes\cdots\otimes x_{k_{m-1}})^{\otimes \ell}
\sim
\cdots
\sim
(x_{k_{2}}\otimes\cdots\otimes x_{k_{1}})^{\otimes \ell}
\sim
-(x_{k_{1}}\otimes\cdots\otimes x_{k_{m}})^{\otimes \ell}.
$$

This tells us that if $\omega_{m,d}$ is the set of all cycle families of words of length $m$ and cycle length $m$ in $x_1,\hdots,x_d$,       $$\bar{\mathfrak{m}}^{\otimes w} / (x_{n_{1}}\otimes {\cdots} \otimes x_{n_{w}} \sim (-1)^{w-1}x_{n_{w}}\otimes x_{n_{1}}\otimes {\cdots} \otimes x_{n_{w-1}}) \cong
$$
$$
\Bigg(\displaystyle\bigoplus_{\substack{m \mid w \\ m \equiv w \Mod{2}}}\displaystyle\bigoplus_{\omega_{m,d}} \mathbb{Q}\Bigg) \bigoplus \Bigg(        \displaystyle\bigoplus_{\substack{m \mid w \\ m \not\equiv w \Mod{2}}}\displaystyle\bigoplus_{\omega_{m,d}} \mathbb{Q}[x]/[x\sim-x]\Bigg) \cong \displaystyle\bigoplus_{\substack{m \mid w \\ m \equiv w \Mod{2}}}\displaystyle\bigoplus_{\omega_{m,d}} \mathbb{Q}. $$
The number of elements in each $\omega_{m,d}$ is
$$
\hspace{5mm} \frac{\sum_{i \vert m}\mu(m/i) d^i}{m}, $$
where $\mu$ is the M\"obius function defined as
$$\mu(n) = \begin{dcases}
1 & \textrm{ when $n$ is a square-free positive integer with an even number of prime factors} \\
-1 & \textrm{ when $n$ is a square-free positive integer with an odd number of prime factors} \\
0 & \textrm{ when $n$ has a squared prime factor} \\
\end{dcases}$$
since we need to count all the $d^{m}$ words of length $m$ in $x_{1}, \ldots, x_{d}$, but subtract $d^{m/p}$ for words which are repeats of length $m/p$ for $p|m$, correct by adding $d^{m/pq}$ for all words which are repeats of words of length $m/pq$ for $p,q$ distinct primes dividing $m$, etc. 

\begin{corollary}
Let $A = \mathbb{Q}[x_1,...,x_d]/\mathfrak{m}^2$ where $\mathfrak{m}$ is the ideal $(x_1,\hdots,x_d)$. Then for $w = 0$
\begin{center}
$
HC_{n}^{(0)}(A) \cong 
    \begin{dcases}
        \mathbb{Q}  & n \textrm{ even and } n \geq 0  \\
        0 & \textrm{else} \\
    \end{dcases}
$
\end{center}
$\\*$ and for $w > 0$

\begin{center}
$
HC_{n}^{(w)}(A) \cong 
    \begin{dcases} 
       \displaystyle\bigoplus_{\substack{m \mid w \\ m \equiv w \Mod{2}}}\displaystyle\bigoplus_{\omega_{m,d}} \mathbb{Q}& n = w-1 \\
        0  & \textrm{else}.\> 
    \end{dcases}
$
\end{center}
\end{corollary}

\section{Cyclic Homology of $\mathbb{Z}[x_1,x_2,\hdots,x_d]/\mathfrak{m}^2$}

Now we use Equation (\ref{HCNW2}) to compute $HC_{n}^{(w)}(A)$ for $k = \mathbb{Z}$. Let $\bar{\mathfrak{m}}_{\mathbb{Q}} = \mathfrak{m} / \mathfrak{m}^{2}$ when $A = \mathbb{Q}[x_1,...,x_d]/\mathfrak{m}^2$ and $\mathfrak{m}$ is the ideal $(x_1,\hdots,x_d)$. Let $\bar{\mathfrak{m}}_{\mathbb{Z}} = \mathfrak{m} / \mathfrak{m}^{2}$ when $A = \mathbb{Z}[x_1,\hdots,x_d]/\mathfrak{m}^2$ and $\mathfrak{m}$ is the ideal $(x_1,\hdots,x_d)$. So $\bar{\mathfrak{m}}_{\mathbb{Q}}$ is the free $\mathbb{Q}$-vector space on $[x_{1}]$, $[x_{2}],\hdots,[x_{d}]$ which by abuse of notation we call $x_{1}$, $x_{2},\hdots , x_{d}$ and $\bar{\mathfrak{m}}_{\mathbb{Z}}$ is the free $\mathbb{Z}$-module on these generators. Note that because $\alpha$ acts as $t$, the fact that $\mathbb{Q}$ is projective over $\mathbb{Q}[C_{w}]$ and thus has no higher Tor gives us
\begin{equation}\label{hadtoremake1}
\ker\big((1-t_{w-1}):\bar{\mathfrak{m}}_{\mathbb{Q}}^{\otimes w}\rightarrow \bar{\mathfrak{m}}_{\mathbb{Q}}^{\otimes w}\big) \cong \im(N:\bar{\mathfrak{m}}_{\mathbb{Q}}^{\otimes w}\rightarrow \bar{\mathfrak{m}}_{\mathbb{Q}}^{\otimes w})
\end{equation}
\begin{equation}\label{hadtoremake2}
\im\big((1-t_{w-1}):\bar{\mathfrak{m}}_{\mathbb{Q}}^{\otimes w}\rightarrow \bar{\mathfrak{m}}_{\mathbb{Q}}^{\otimes w}\big) \cong \ker(N:\bar{\mathfrak{m}}_{\mathbb{Q}}^{\otimes w}\rightarrow \bar{\mathfrak{m}}_{\mathbb{Q}}^{\otimes w}) \> .
\end{equation}
By Equation (\ref{complexmbar}), $Tor_{n}^{\mathbb{Z}[C_{w}]}(\mathbb{Z},\bar{\mathfrak{m}}^{\otimes w})$ is the homology of the complex
\begin{equation}\label{TORHomComplex}    
\cdots\xrightarrow{1-t_{w-1}} \bar{\mathfrak{m}}_{\mathbb{Z}}^{\otimes w}\xrightarrow{N_{w-1}} \bar{\mathfrak{m}}_{\mathbb{Z}}^{\otimes w}\xrightarrow{1-t_{w-1}} \bar{\mathfrak{m}}_{\mathbb{Z}}^{\otimes w}\xrightarrow{N_{w-1}} \bar{\mathfrak{m}}_{\mathbb{Z}}^{\otimes w}\xrightarrow{1-t_{w-1}}\bar{\mathfrak{m}}_{\mathbb{Z}}^{\otimes w}\xrightarrow{} 0 \>.
\end{equation}
Therefore, from Theorem \ref{HCNW3}
\begin{equation}\label{ccz1}
HC_{n}^{(0)}(\mathbb{Z}[x_1,x_2,\hdots,x_d]/\mathfrak{m}^2) \cong 
    \begin{dcases}
        \mathbb{Z} &  n \textrm{ even and } n \geq 0  \\
        0 & n \textrm{ else} \\
    \end{dcases}
\end{equation}
and for $w > 0$

\begin{equation}\label{ccz2}
HC_{n}^{(w)}(\mathbb{Z}[x_1,x_2,\hdots,x_d]/\mathfrak{m}^2) 
\cong 
    \begin{dcases} 
       \>\>\>\>\>\>\>\>\>\>\>\>\>\>\>\>\>\>\>\>\>\>\>\>\> 0  & n \leq w-2 \\
       \\
       \frac{\bar{\mathfrak{m}}_{\mathbb{Z}}^{\otimes w}}{\im\big((1-t_{w-1}):\bar{\mathfrak{m}}_{\mathbb{Z}}^{\otimes w}\rightarrow \bar{\mathfrak{m}}_{\mathbb{Z}}^{\otimes w}\big)} & n = w-1 \\
       \\
        \frac{\ker\big((1-t_{w-1}):\bar{\mathfrak{m}}_{\mathbb{Z}}^{\otimes w}\rightarrow \bar{\mathfrak{m}}_{\mathbb{Z}}^{\otimes w}\big)}{\im(N_{w-1}:\bar{\mathfrak{m}}_{\mathbb{Z}}^{\otimes w}\rightarrow \bar{\mathfrak{m}}_{\mathbb{Z}}^{\otimes w})} & n = w + 2i, \hspace{11.5mm} i\geq0 \\
        \\
       \frac{\ker{(N_{w-1}:\bar{\mathfrak{m}}_{\mathbb{Z}}^{\otimes w}\rightarrow \bar{\mathfrak{m}}_{\mathbb{Z}}^{\otimes w}})}{\im\big((1-t_{w-1}):\bar{\mathfrak{m}}_{\mathbb{Z}}^{\otimes w}\rightarrow \bar{\mathfrak{m}}_{\mathbb{Z}}^{\otimes w}\big)} & n = w + 1 + 2i, \hspace{5mm} i\geq0 \> . \\
    \end{dcases}
\end{equation}
Equations (\ref{hadtoremake1}) and (\ref{hadtoremake2}) gives us 
\begin{center}
$\ker\big((1-t_{w-1}): \bar{\mathfrak{m}}_{\mathbb{Z}}^{\otimes w}\rightarrow \bar{\mathfrak{m}}_{\mathbb{Z}}^{\otimes w}\big) \cong \big(\ker\big((1-t_{w-1}): \bar{\mathfrak{m}}_{\mathbb{Q}}^{\otimes w}\rightarrow \bar{\mathfrak{m}}_{\mathbb{Q}}^{\otimes w}\big)\big) \cap \bar{\mathfrak{m}}_{\mathbb{Z}}^{\otimes w}, $ 

\medskip

$\ker(N_{w-1}: \bar{\mathfrak{m}}_{\mathbb{Z}}^{\otimes w}\rightarrow \bar{\mathfrak{m}}_{\mathbb{Z}}^{\otimes w}) \cong \big(\ker(N_{w-1}: \bar{\mathfrak{m}}_{\mathbb{Q}}^{\otimes w}\rightarrow \bar{\mathfrak{m}}_{\mathbb{Q}}^{\otimes w})\big) \cap \bar{\mathfrak{m}}_{\mathbb{Z}}^{\otimes w}. $
\end{center}
$\\*$So we can deduce
\begin{proposition}\label{applykertozandq}

Let $A = \mathbb{Z}[x_1,x_2,\hdots,x_d]/\mathfrak{m}^2$ where $\mathfrak{m}$ is the ideal $(x_1,x_2,\hdots,x_d)$ and $\bar{\mathfrak{m}} = \mathfrak{m} / \mathfrak{m}^{2}$.
Then

\begin{center}
$
HC_{n}^{(0)}(A) \cong 
    \begin{dcases}
        \mathbb{Z} & n \textrm{ even and } n \geq 0 \\
        0 & \textrm{else}\\
    \end{dcases}
$
\end{center}
and for $w > 0$

\begin{center}
$
HC_{n}^{(w)}(A) \cong 
    \begin{dcases} 
       \>\>\>\>\>\>\>\>\>\>\>\>\>\>\>\>\>\>\>\>\>\>\>\>\> 0  & n \leq w-2 \\
       \\
       \frac{\bar{\mathfrak{m}}_{\mathbb{Z}}^{\otimes w}}{\im\big((1-t_{w-1}):\bar{\mathfrak{m}}_{\mathbb{Z}}^{\otimes w}\rightarrow \bar{\mathfrak{m}}_{\mathbb{Z}}^{\otimes w}\big)} & n = w-1 \\
       \\
        \frac{\big(\im(N_{w-1}: \bar{\mathfrak{m}}_{\mathbb{Q}}^{\otimes w}\rightarrow \bar{\mathfrak{m}}_{\mathbb{Q}}^{\otimes w})\big) \cap \bar{\mathfrak{m}}_{\mathbb{Z}}^{\otimes w}}{\im(N_{w-1}:\bar{\mathfrak{m}}_{\mathbb{Z}}^{\otimes w}\rightarrow \bar{\mathfrak{m}}_{\mathbb{Z}}^{\otimes w})} & n = w + 2i, \hspace{11.5mm}i\geq0  \\
        \\
       \frac{\big(\im\big((1-t_{w-1}): \bar{\mathfrak{m}}_{\mathbb{Q}}^{\otimes w}\rightarrow \bar{\mathfrak{m}}_{\mathbb{Q}}^{\otimes w}\big)\big) \cap \bar{\mathfrak{m}}_{\mathbb{Z}}^{\otimes w}}{\im\big((1-t_{w-1}):\bar{\mathfrak{m}}_{\mathbb{Z}}^{\otimes w}\rightarrow \bar{\mathfrak{m}}_{\mathbb{Z}}^{\otimes w}\big)} & n = w + 1 + 2i, \hspace{5mm}i\geq0\> . \\
    \end{dcases}
$
\end{center}
\end{proposition}
$\\*$The next three subsections will explicitly compute the pieces of Proposition \ref{applykertozandq}.





\subsection{Calculating $HC_{w+2i}^{(w)}$ for $i\geq0$}\hfill

From Proposition \ref{applykertozandq}, we know
$$
HC_{w+2i}^{(w)}(\mathbb{Z}[x_1,x_2,\hdots,x_d]/\mathfrak{m}^2)\cong \frac{\big(\im(N_{w-1}: \bar{\mathfrak{m}}_{\mathbb{Q}}^{\otimes w}\rightarrow \bar{\mathfrak{m}}_{\mathbb{Q}}^{\otimes w})\big) \cap \bar{\mathfrak{m}}_{\mathbb{Z}}^{\otimes w}}{\im(N_{w-1}:\bar{\mathfrak{m}}_{\mathbb{Z}}^{\otimes w}\rightarrow \bar{\mathfrak{m}}_{\mathbb{Z}}^{\otimes w})}.
$$
The module $\bar{\mathfrak{m}}_{\mathbb{Z}}^{\otimes w}$ is freely spanned over $\mathbb{Z}$ by elements $x_{j_{1}}\otimes\cdots\otimes x_{j_{w}}$ where $j_{i} \in \{1,\hdots,d\}$ for $i\in \{1,\hdots,w\}$. Define a new map
$T_{w-1}(x_{j_{1}}\otimes x_{j_{2}}\otimes\cdots\otimes x_{j_{w}}) = x_{j_{w}}\otimes x_{j_{1}}\otimes\cdots\otimes x_{j_{w-1}}.$
$\\*$Note that
\begin{equation}\label{tnandTn}
t_{n}(a_{0}\otimes a_{1}\otimes \cdots\otimes a_{n}) = (-1)^{n}(a_{n}\otimes a_{0}\otimes a_{1} \otimes \cdots \otimes a_{n-1}) = (-1)^{n}T_{n}(a_{0}\otimes a_{1}\otimes \cdots\otimes a_{n})
\end{equation} 
\begin{definition}\label{cl}
Define the \textbf{cycle length} of a tensor monomial $x_{j_{1}}\otimes\cdots\otimes x_{j_{w}}$ to be the smallest $0<m$ such that $T_{w-1}^{m}(x_{j_{1}}\otimes\cdots\otimes x_{j_{w}}) = x_{j_{1}}\otimes\cdots\otimes x_{j_{w}}$.
\end{definition}


\begin{definition}
Define the \textbf{cycle family} of a tensor monomial $x_{j_{1}}\otimes\cdots\otimes x_{j_{w}}$ to be the set of all the $T_{w-1}^{n}(x_{j_{1}}\otimes\cdots\otimes x_{j_{w}})$ for  $n \in \mathbb{N}$.
\end{definition}
All tensor monomials in a cycle family have the same cycle length. Also, $t_{w-1}$ and $N_{w-1}$ send elements of a cycle family to sums of elements of the same cycle family. Therefore, we can break the calculation down  by cycle families. 
\begin{lemma}\label{paritydifferent}
Consider a tensor monomial $x_{j_{1}}\otimes\cdots\otimes x_{j_{w}}$ of cycle length $m$ such that $m$ is not of the same parity as $w$. Then $N_{w-1}(x_{j_{1}}\otimes\cdots\otimes x_{j_{w}}) = 0$.
\end{lemma}

\begin{proof}
First, note that since $T_{w-1}^{(w)}$ is the identity map, $m$ must divide $w$. Odd numbers only have odd divisors, so the only case we need to consider is when $w$ is even and $m$ is odd. Write $w = 2\ell m$ for some positive $\ell\in \mathbb{Z}$. A tensor monomial $x_{j_{1}}\otimes\cdots\otimes x_{j_{w}}$ in $\bar{\mathfrak{m}}_{\mathbb{Z}}^{\otimes w}$ of cycle length $m$ must be of the form $(x_{k_{1}}\otimes \cdots \otimes x_{k_{m}})^{\otimes \ell}$ where $k_{1},\ldots, k_{m} \in \{1,2, \ldots, d\}$ are such that $x_{k_{1}}\otimes\cdots\otimes x_{k_{m}}$ has cycle length $m$. Then
$$
N_{w-1}(x_{k_{1}}\otimes\cdots\otimes x_{k_{m}}\otimes \cdots \otimes x_{k_{1}}\otimes\cdots\otimes x_{k_{m}}) = \sum\limits_{i=0}^{w-1} t_{w-1}^{i}(x_{k_{1}}\otimes\cdots\otimes x_{k_{m}}\otimes \cdots \otimes x_{k_{1}}\otimes\cdots\otimes x_{k_{m}})
$$
$$
  = \sum\limits_{i=0}^{w-1} (-1)^{i \cdot (w-1)} T_{w-1}^{i}(x_{k_{1}}\otimes\cdots\otimes x_{k_{m}}\otimes \cdots \otimes x_{k_{1}}\otimes\cdots\otimes x_{k_{m}}).
$$
Since the tensor monomial has cycle length $m$, if $ a \equiv b \Mod{m}$  then 
$$
T_{w-1}^{a}(x_{k_{1}}\otimes\cdots\otimes x_{k_{m}}\otimes \cdots \otimes x_{k_{1}}\otimes\cdots\otimes x_{k_{m}}) = T_{w-1}^{b}(x_{k_{1}}\otimes\cdots\otimes x_{k_{m}}\otimes \cdots \otimes x_{k_{1}}\otimes\cdots\otimes x_{k_{m}}).
$$
The $i^{th}$ summand, $i = jm + a$ has sign $(-1)^{j}(-1)^{a}$, so this sum becomes
$$
\sum\limits_{j=0}^{2\ell-1} (-1)^{j} \Bigg( \sum\limits_{i=0}^{m-1} (-1)^{i} T_{w-1}^{i}(x_{k_{1}}\otimes\cdots\otimes x_{k_{m}}\otimes \cdots \otimes x_{k_{1}}\otimes\cdots\otimes x_{k_{m}})\Bigg) = 0.
$$
\end{proof}
\begin{lemma}\label{paritysame}
Consider a tensor monomial $x_{j_{1}}\otimes\cdots\otimes x_{j_{w}}$ of cycle length $m$ such that $m$ is the same parity as $w$ and $w=\ell \cdot m$. Then $N_{w-1}(x_{j_{1}}\otimes\cdots\otimes x_{j_{w}}) = \ell \cdot ( 1 + t_{w-1} + t_{w-1}^{2} + \cdots + t_{w-1}^{m-1}) (x_{j_{1}}\otimes\cdots\otimes x_{j_{w}})$.\end{lemma}\begin{proof}
If $w = \ell\cdot m$ for $\ell$, $m$ odd, then $t_{w-1} = T_{w-1}$ so
$$
 N_{w-1}(x_{k_{1}}\otimes\cdots\otimes x_{k_{m}}\otimes \cdots \otimes x_{k_{1}}\otimes\cdots\otimes x_{k_{m}})  
  = \sum\limits_{i=0}^{w-1} t_{w-1}^{i}(x_{k_{1}}\otimes\cdots\otimes x_{k_{m}}\otimes \cdots \otimes x_{k_{1}}\otimes\cdots\otimes x_{k_{m}}) 
$$
$$= \ell  \sum\limits_{i=0}^{m-1} t_{w-1}^{i}(x_{k_{1}}\otimes\cdots\otimes x_{k_{m}}\otimes \cdots \otimes x_{k_{1}}\otimes\cdots\otimes x_{k_{m}}) .
$$
If $w = \ell\cdot m$ for $m$ even,
$$
 N_{w-1}(x_{k_{1}}\otimes\cdots\otimes x_{k_{m}}\otimes \cdots \otimes x_{k_{1}}\otimes\cdots\otimes x_{k_{m}})  
  = \sum\limits_{i=0}^{w-1} t_{w-1}^{i}(x_{k_{1}}\otimes\cdots\otimes x_{k_{m}}\otimes \cdots \otimes x_{k_{1}}\otimes\cdots\otimes x_{k_{m}}) 
$$
$$
  = \sum\limits_{i=0}^{w-1} (-1)^{i \cdot (w-1)} T_{w-1}^{i}(x_{k_{1}}\otimes\cdots\otimes x_{k_{m}}\otimes \cdots \otimes x_{k_{1}}\otimes\cdots\otimes x_{k_{m}})
  $$
  $$
 = \sum\limits_{i=0}^{w-1} (-1)^{i} T_{w-1}^{i}(x_{k_{1}}\otimes\cdots\otimes x_{k_{m}}\otimes \cdots \otimes x_{k_{1}}\otimes\cdots\otimes x_{k_{m}})
$$
$$
= \ell  \sum\limits_{i=0}^{m-1} (-1)^{i} T_{w-1}^{i}(x_{k_{1}}\otimes\cdots\otimes x_{k_{m}}\otimes \cdots \otimes x_{k_{1}}\otimes\cdots\otimes x_{k_{m}})
$$
$$
= \ell   \sum\limits_{i=0}^{m-1} t_{w-1}^{i}(x_{k_{1}}\otimes\cdots\otimes x_{k_{m}}\otimes \cdots \otimes x_{k_{1}}\otimes\cdots\otimes x_{k_{m}}) .
$$ \end{proof}
\begin{lemma}\label{1}
Consider the family of rings $k[x_1,\hdots,x_d]/\mathfrak{m}^2$, where $\mathfrak{m}$ is the ideal $(x_1,\hdots,x_d)$ and $\bar{\mathfrak{m}} = \mathfrak{m} / \mathfrak{m}^{2}$. For any positive integer $w$,

        $$
        \dfrac{\big(\im(N_{w-1}: \bar{\mathfrak{m}}_{\mathbb{Q}}^{\otimes w}\rightarrow \bar{\mathfrak{m}}_{\mathbb{Q}}^{\otimes w})\big) \cap \bar{\mathfrak{m}}_{\mathbb{Z}}^{\otimes w}}{\im(N_{w-1}:\bar{\mathfrak{m}}_{\mathbb{Z}}^{\otimes w}\rightarrow \bar{\mathfrak{m}}_{\mathbb{Z}}^{\otimes w})} \cong \displaystyle\bigoplus_{\substack{m \mid w \\ m \equiv w \Mod{2}}}\displaystyle\bigoplus_{\omega_{m,d}} \mathbb{Z}/\big(\tfrac{w}{m}\big)
        $$
where $\omega_{m,d} = \{\textrm{all cycle families of words of length $m$ and cycle length $m$ in $x_1,\hdots,x_d$}\}$.\end{lemma}
$\\*$Recall the number of elements in $\omega_{m,d}$ is $\frac{1}{m}\sum_{i \vert m}\mu(m/i) d^{i}$.
\begin{proof}
For cycle families of cycle length $m$ where $w = \ell \cdot m$, we only need to consider the case where $m \equiv w \Mod{2}$ by  Lemma \ref{paritydifferent}. There are $m$ tensor monomials in its cycle family. By Lemma \ref{paritysame}, 
$$
 N_{w-1}\big((x_{k_{1}}\otimes\cdots\otimes x_{k_{m}})^{\otimes \ell}\big)=\ell  \sum\limits_{j=0}^{m-1} t_{w-1}^{j}\big((x_{k_{1}}\otimes\cdots\otimes x_{k_{m}})^{\otimes \ell}\big) .
$$
For any choice of coefficients $a_{i} \in \mathbb{Q}$ for $i \in \{1,2,\hdots,m\}$,
$$
N_{w-1}\big(a_{1}(x_{k_{1}}\otimes\cdots\otimes x_{k_{m}})^{\otimes \ell}+a_{2}(x_{k_{m}}\otimes\cdots\otimes x_{k_{m-1}})^{\otimes \ell} +\cdots+ a_{m-1}(x_{k_{3}}\otimes\cdots\otimes x_{k_{2}})^{\otimes \ell}+a_{m}(x_{k_{2}}\otimes\cdots\otimes x_{k_{1}})^{\otimes \ell}\big)
$$
$$
= \ell \cdot a_{1} \sum\limits_{j=0}^{m-1} t_{w-1}^{j}\big((x_{k_{1}}\otimes\cdots\otimes x_{k_{m}})^{\otimes \ell}\big) + \ldots + \ell \cdot a_{m} \sum\limits_{j=0}^{m-1} t_{w-1}^{j}\big((x_{k_{2}}\otimes\cdots\otimes x_{k_{1}})^{\otimes \ell}\big).
$$
$\\*$\textbf{Case 1:} If $w$ and $m$ are odd, 
$$
\ell \cdot a_{1} \sum\limits_{j=0}^{m-1} t_{w-1}^{j}\big((x_{k_{1}}\otimes\cdots\otimes x_{k_{m}})^{\otimes \ell}\big) + \ldots + \ell \cdot a_{m} \sum\limits_{j=0}^{m-1} t_{w-1}^{j}\big((x_{k_{2}}\otimes\cdots\otimes x_{k_{1}})^{\otimes \ell}\big)  
$$
$$
=\ell \cdot \Bigg(\sum\limits_{i=1}^{m} a_{i}\Bigg) \cdot \sum\limits_{j=0}^{m-1} t_{w-1}^{j}\big((x_{k_{1}}\otimes\cdots\otimes x_{k_{m}})^{\otimes \ell}\big). 
$$
Note that $\ell \cdot \bigg(\sum\limits_{i=1}^{m} a_{i}\bigg) \cdot \sum\limits_{j=0}^{m-1} t_{w-1}^{j}\big((x_{k_{1}}\otimes\cdots\otimes x_{k_{m}})^{\otimes \ell}\big) \in \bar{\mathfrak{m}}_{\mathbb{Z}}^{\otimes w}$ if and only if $\sum\limits_{i=1}^{m} a_{i} \in \dfrac{1}{\ell} \mathbb{Z}$. Therefore, $\big(\im(N_{w-1}: \bar{\mathfrak{m}}_{\mathbb{Q}}^{\otimes w}\rightarrow \bar{\mathfrak{m}}_{\mathbb{Q}}^{\otimes w})\big) \cap \bar{\mathfrak{m}}_{\mathbb{Z}}^{\otimes w} = \mathbb{Z}  \cdot \sum\limits_{j=0}^{m-1} t_{w-1}^{j}\big((x_{k_{1}}\otimes\cdots\otimes x_{k_{m}})^{\otimes \ell}\big) $. But 
$\im(N_{w-1}:\bar{\mathfrak{m}}_{\mathbb{Z}}^{\otimes w}\rightarrow \bar{\mathfrak{m}}_{\mathbb{Z}}^{\otimes w})  = \ell \cdot \mathbb{Z} \cdot \sum\limits_{j=0}^{m-1} t_{w-1}^{j}\big((x_{k_{1}}\otimes\cdots\otimes x_{k_{m}})^{\otimes \ell}\big)  $. So

$$
        \dfrac{\big(\im(N_{w-1}: \bar{\mathfrak{m}}_{\mathbb{Q}}^{\otimes w}\rightarrow \bar{\mathfrak{m}}_{\mathbb{Q}}^{\otimes w})\big) \cap \bar{\mathfrak{m}}_{\mathbb{Z}}^{\otimes w}}{\im(N_{w-1}:\bar{\mathfrak{m}}_{\mathbb{Z}}^{\otimes w}\rightarrow \bar{\mathfrak{m}}_{\mathbb{Z}}^{\otimes w})} \cong \mathbb{Z} / \ell \mathbb{Z}.
$$

$\\*$ \textbf{Case 2:} If $w$ and $m$ are even, 
$$
\ell \cdot a_{1} \sum\limits_{j=0}^{m-1} t_{w-1}^{j}\big((x_{k_{1}}\otimes\cdots\otimes x_{k_{m}})^{\otimes \ell}\big) + \ldots + \ell \cdot a_{m} \sum\limits_{j=0}^{m-1} t_{w-1}^{j}\big((x_{k_{2}}\otimes\cdots\otimes x_{k_{1}})^{\otimes \ell}\big)  
$$
$$
=\ell \cdot \Bigg(\sum\limits_{i=1}^{m} (-1)^{i+1}a_{i}\Bigg) \cdot \sum\limits_{j=0}^{m-1} t_{w-1}^{j}\big((x_{k_{1}}\otimes\cdots\otimes x_{k_{m}})^{\otimes \ell}\big).
$$
From here, the proof is the same as Case $1$ if we replace $\sum\limits_{i=1}^{m} a_{i}$ with $\sum\limits_{i=1}^{m} (-1)^{i+1}a_{i}$.

\end{proof}
\subsection{Calculating $HC_{w+1+2i}^{(w)}$ for $i\geq0$}\label{im1-tim1-t} \hfill

In Proposition \ref{applykertozandq}, we saw that:
$$
HC_{w+1+2i}^{(w)}(\mathbb{Z}[x_1,x_2,\hdots,x_d]/\mathfrak{m}^2)\cong \frac{(\im((1-t_{w-1}): \bar{\mathfrak{m}}_{\mathbb{Q}}^{\otimes w}\rightarrow \bar{\mathfrak{m}}_{\mathbb{Q}}^{\otimes w})) \cap \bar{\mathfrak{m}}_{\mathbb{Z}}^{\otimes w}}{\im((1-t_{w-1}):\bar{\mathfrak{m}}_{\mathbb{Z}}^{\otimes w}\rightarrow \bar{\mathfrak{m}}_{\mathbb{Z}}^{\otimes w})}.
$$
This section looks at images of tensor monomials of length $w$ and cycle length $m$ under the map $(1-t_{w-1})$.

\begin{lemma}\label{wecl1}
For $w$ even, each cycle family of words of cycle length $m=1$ will contribute a copy of $\mathbb{Z}/2$ to $\dfrac{(\im((1-t_{w-1}): \bar{\mathfrak{m}}_{\mathbb{Q}}^{\otimes w}\rightarrow \bar{\mathfrak{m}}_{\mathbb{Q}}^{\otimes w})) \cap \bar{\mathfrak{m}}_{\mathbb{Z}}^{\otimes w}}{\im((1-t_{w-1}):\bar{\mathfrak{m}}_{\mathbb{Z}}^{\otimes w}\rightarrow \bar{\mathfrak{m}}_{\mathbb{Z}}^{\otimes w})}$.
\end{lemma}
\begin{proof}
For $w$ even and $i \in \{1,2,\hdots,d\}$, 
$$
(1-t_{w-1})(x_{i}^{\otimes w})= x_{i}^{\otimes w} + x_{i}^{\otimes w} = 2(x_{i}^{\otimes w}).
$$
\end{proof}
\begin{lemma}\label{wocl1}
For $w$ odd, each cycle family of words of cycle length $m=1$ will contribute nothing to $\dfrac{(\im((1-t_{w-1}): \bar{\mathfrak{m}}_{\mathbb{Q}}^{\otimes w}\rightarrow \bar{\mathfrak{m}}_{\mathbb{Q}}^{\otimes w})) \cap \bar{\mathfrak{m}}_{\mathbb{Z}}^{\otimes w}}{\im((1-t_{w-1}):\bar{\mathfrak{m}}_{\mathbb{Z}}^{\otimes w}\rightarrow \bar{\mathfrak{m}}_{\mathbb{Z}}^{\otimes w})}$.
\end{lemma}
\begin{proof}
For $w$ even and $i \in \{1,2,\hdots,d\}$, 
$$
(1-t_{w-1})(x_{i}^{\otimes w})= x_{i}^{\otimes w} - x_{i}^{\otimes w} = 0.
$$
\end{proof}

\begin{lemma}\label{woclo}
For any $w>1$ odd, each cycle family of words of cycle length $m>1$ will contribute nothing to $\dfrac{\big(\im\big((1-t_{w-1}): \bar{\mathfrak{m}}_{\mathbb{Q}}^{\otimes w}\rightarrow \bar{\mathfrak{m}}_{\mathbb{Q}}^{\otimes w}\big)\big) \cap \bar{\mathfrak{m}}_{\mathbb{Z}}^{\otimes w}}{\im\big((1-t_{w-1}):\bar{\mathfrak{m}}_{\mathbb{Z}}^{\otimes w}\rightarrow \bar{\mathfrak{m}}_{\mathbb{Z}}^{\otimes w}\big)}$.
\end{lemma}
\begin{proof} Let $w >1$ be odd, and write $w = m \cdot \ell$. Then $m$ must also be odd. If the word is $(x_{k_{1}}\otimes \cdots \otimes x_{k_{m}})^{\otimes \ell}$, then there are $m$ tensor monomials in its cycle family. 

$\\*$Consider the image of the span over $\mathbb{Q}$ of these tensor monomials under $\big( 1-t_{w-1} \big): \bar{\mathfrak{m}}_{\mathbb{Q}}^{\otimes w}\rightarrow \bar{\mathfrak{m}}_{\mathbb{Q}}^{\otimes w}$. Let $a_{i} \in \mathbb{Q}$ for any $i \in \{1,2,\hdots,m\}$. Taking the $m$ words in the cycle family as a basis for $\bar{\mathfrak{m}}_{\mathbb{Q}}^{\otimes w}$, we can express the map $(1-t_{w-1})$ by the matrix

  

\[
\begin{bmatrix}
1 & 0 & \cdots & 0 & -1\\
-1 & 1 & \cdots& 0 & 0\\
0 & -1 & \cdots & 0 & 0 \\
\vdots & \vdots & \ddots & \vdots & \vdots\\
0 & 0& \cdots & 1&0 \\
0 & 0& \cdots & -1& 1 \\
\end{bmatrix}_{m\times m}.
\]

$\\*$If we row reduce (which can be done over $\mathbb{Z}$), then we get the matrix 

\[
\begin{bmatrix}
1 & 0 & \cdots & 0 & -1\\
0& 1 & \cdots& 0 & -1\\
0 & 0& \cdots & 0 & -1\\
\vdots & \vdots & \ddots & \vdots & \vdots\\
0 & 0& \cdots & 1&-1 \\
0 & 0& \cdots & 0& 0 \\
\end{bmatrix}_{m\times m}.
\]

$\\*$So the first $(m-1)$-columns of the original $m \times m $ matrix are a basis for $\im(1-t_{w-1})$, and the question is: what can be said about $\alpha_{1}, \alpha_{2},\hdots,\alpha_{m-1}$ in the following equation if $b_{1},b_{2},\hdots,b_{m}$ are all integers?

 \begin{align}\label{oddspan2}
          \alpha_{1}\begin{bmatrix}
           1 \\  
           -1\\
           0\\
           \vdots \\
           0\\
           0\\
          \end{bmatrix}_{m\times 1} +
          \alpha_{2}\begin{bmatrix}
           0 \\  
           1\\
           -1\\
           \vdots \\
           0\\
           0\\
          \end{bmatrix}_{m\times 1}+ \cdots +
          \alpha_{m-1}\begin{bmatrix}
           0 \\  
           0\\
            0\\
           \vdots \\
           1\\
           -1\\
          \end{bmatrix}_{m\times 1} =
          \begin{bmatrix}
           b_{1} \\  
           b_{2} \\
           b_{3}\\
           \vdots \\
           b_{m-1} \\
           b_{m} \\
          \end{bmatrix}_{m\times 1}
  \end{align}
If $b_{1},b_{2},\hdots,b_{m}$ are all integers, then $\alpha_{1}, \alpha_{2},\hdots,\alpha_{m-1}$ are also all integers. Therefore, a cycle family of tensor monomials of length $w>2$ for $w$ odd and with cycle length $m>2$ does not generate anything in $\frac{(\im((1-t_{w-1}): \bar{\mathfrak{m}}_{\mathbb{Q}}^{\otimes w}\rightarrow \bar{\mathfrak{m}}_{\mathbb{Q}}^{\otimes w})) \cap \bar{\mathfrak{m}}_{\mathbb{Z}}^{\otimes w}}{\im((1-t_{w-1}):\bar{\mathfrak{m}}_{\mathbb{Z}}^{\otimes w}\rightarrow \bar{\mathfrak{m}}_{\mathbb{Z}}^{\otimes w})}$.

\end{proof}

\begin{lemma}\label{wecle}
For any $w>1$ even, each cycle family of words of cycle length $m$ also even will contribute nothing to $\dfrac{\big(\im\big((1-t_{w-1}): \bar{\mathfrak{m}}_{\mathbb{Q}}^{\otimes w}\rightarrow \bar{\mathfrak{m}}_{\mathbb{Q}}^{\otimes w}\big)\big) \cap \bar{\mathfrak{m}}_{\mathbb{Z}}^{\otimes w}}{\im\big((1-t_{w-1}):\bar{\mathfrak{m}}_{\mathbb{Z}}^{\otimes w}\rightarrow \bar{\mathfrak{m}}_{\mathbb{Z}}^{\otimes w}\big)}$.
\end{lemma}
\begin{proof}  $w=m\cdot \ell$. Then a monomial
$(x_{k_{1}}\otimes\cdots\otimes x_{k_{m}})^{\otimes \ell}$ with cycle length $m$ has $m$ monomials in its cycle family. Consider the image of the span over $\mathbb{Q}$ of these words under 
$(1-t_{w-1}): \bar{\mathfrak{m}}_{\mathbb{Q}}^{\otimes w}\rightarrow \bar{\mathfrak{m}}_{\mathbb{Q}}^{\otimes w}$. Taking the $m$ words in the cycle family as a basis for $\bar{\mathfrak{m}}_{\mathbb{Q}}^{\otimes w}$, we can express the map $(1-t_{w-1})$ by the matrix
\[
\begin{bmatrix}
1 & 0 & \cdots & 0 & 1\\
1 & 1 & \cdots& 0 & 0\\
0 & 1 & \cdots & 0 & 0 \\
\vdots & \vdots & \ddots & \vdots & \vdots\\
0 & 0& \cdots & 1&0 \\
0 & 0& \cdots & 1& 1 \\
\end{bmatrix}_{m\times m}.
\]

$\\*$We get the following row-reduced version of the matrix (since $m$ is even)

\[
\begin{bmatrix}
1 & 0 &0& 0&\cdots & 0& 0& 0 & 1\\
0& 1 & 0&0&\cdots& 0 & 0& 0& -1\\
0 & 0& 1&0&\cdots & 0& 0& 0 & 1\\
0 & 0& 0&1&\cdots & 0& 0& 0 & -1\\
\vdots & \vdots & \vdots & \vdots & \ddots& \vdots & \vdots& \vdots &\vdots \\
0 & 0& 0&0&\cdots & 1&0& 0& 1\\
0 & 0& 0&0&\cdots & 0& 1& 0 & -1\\
0 & 0& 0&0&\cdots & 0& 0& 1&1 \\
0 & 0& 0&0&\cdots & 0& 0& 0& 0 \\
\end{bmatrix}_{m\times m}.
\]
So the first $m-1$ columns of the original matrix span the image of $(1-t_{w-1})$, and the question is: what can be said about $\alpha_{1}, \alpha_{2},\hdots,\alpha_{m-1}$ in the following equation if $b_{1},b_{2},\hdots,b_{m}$ are all integers?

 \begin{align}\label{oddspan4}
          \alpha_{1}\begin{bmatrix}
           1 \\  
           1\\
           0\\
           \vdots \\
           0\\
           0\\
          \end{bmatrix}_{m\times 1} +
          \alpha_{2}\begin{bmatrix}
           0 \\  
           1\\
           1\\
           \vdots \\
           0\\
           0\\
          \end{bmatrix}_{m\times 1}+ \cdots +
          \alpha_{m-1}\begin{bmatrix}
           0 \\  
           0\\
           0\\
           \vdots \\
           1\\
           1\\
          \end{bmatrix}_{m\times 1} =
          \begin{bmatrix}
           b_{1} \\  
           b_{2} \\
           b_{3}\\
           \vdots \\
           b_{m-1} \\
           b_{m} \\
          \end{bmatrix}_{m\times 1}
  \end{align}
Again, if  $b_{1},b_{2},\hdots,b_{m}$ are all integers, then $\alpha_{1}, \alpha_{2},\hdots,\alpha_{m-1}$ are also all integers. Therefore, a cycle family of monomials of length $w>2$ and $w$ even with cycle length $m$ also even and $m>2$ does not generate anything in $\frac{(\im((1-t_{w-1}): \bar{\mathfrak{m}}_{\mathbb{Q}}^{\otimes w}\rightarrow \bar{\mathfrak{m}}_{\mathbb{Q}}^{\otimes w})) \cap \bar{\mathfrak{m}}_{\mathbb{Z}}^{\otimes w}}{\im((1-t_{w-1}):\bar{\mathfrak{m}}_{\mathbb{Z}}^{\otimes w}\rightarrow \bar{\mathfrak{m}}_{\mathbb{Z}}^{\otimes w})}$.
\end{proof}
\begin{lemma}\label{weclo}
For any $w>1$ even, each cycle family of words of cycle length $m>1$ where $m$ is odd  will contribute a copy of $\mathbb{Z}/2$ to $\dfrac{(\im((1-t_{w-1}): \bar{\mathfrak{m}}_{\mathbb{Q}}^{\otimes w}\rightarrow \bar{\mathfrak{m}}_{\mathbb{Q}}^{\otimes w})) \cap \bar{\mathfrak{m}}_{\mathbb{Z}}^{\otimes w}}{\im((1-t_{w-1}):\bar{\mathfrak{m}}_{\mathbb{Z}}^{\otimes w}\rightarrow \bar{\mathfrak{m}}_{\mathbb{Z}}^{\otimes w})}$.
\end{lemma}
\begin{proof}
Here, the matrix of $(1-t_{w-1})$ is

\[ M=
\begin{bmatrix}
1 & 0 & \cdots & 0 & 1\\
1 & 1 & \cdots& 0 & 0\\
0 & 1 & \cdots & 0 & 0 \\
\vdots & \vdots & \ddots & \vdots & \vdots\\
0 & 0& \cdots & 1&0 \\
0 & 0& \cdots & 1& 1 \\
\end{bmatrix}_{m\times m.}
\]

$\\*$This matrix has full rank. So the question is: what can be said about $\alpha_{1}, \alpha_{2},\hdots,\alpha_{m-1}, \alpha_{m}$ in the following equation if $b_{1},b_{2},\hdots,b_{m}$ are all integers?

 \begin{align}\label{evenoddspan}
          \alpha_{1}\begin{bmatrix}
           1 \\  
           1\\
           0\\
           \vdots \\
           0\\
           0\\
          \end{bmatrix}_{m\times 1} +
          \alpha_{2}\begin{bmatrix}
           0 \\  
           1\\
           1\\
           \vdots \\
           0\\
           0\\
          \end{bmatrix}_{m\times 1}+ \cdots +
          \alpha_{m-1}\begin{bmatrix}
           0 \\  
           0\\
           0\\
           \vdots \\
           1\\
           1\\
          \end{bmatrix}_{m\times 1} +
          \alpha_{m}\begin{bmatrix}
           1 \\  
           0\\
           0\\
           \vdots \\
           0\\
           1\\
          \end{bmatrix}_{m\times 1} =
          \begin{bmatrix}
           b_{1} \\  
           b_{2} \\
           b_{3}\\
           \vdots \\
           b_{m-1} \\
           b_{m} \\
          \end{bmatrix}_{m\times 1}
  \end{align}

$\\*$If we row reduce the associated matrix, we get the following matrix. All of the signs flip on the terms from line to line except one sign that stays the same. The underlining emphasizes the pattern.
\[
\left[
\begin{array}{ccccccccc|c}
1 & 0 &0& 0&\cdots & 0& 0& 0 & 0 & \frac{1}{2}(-b_{m}+b_{m-1} - b_{m-2} +b_{m-3} \cdots -b_{5}+b_{4}-b_{3}\underline{+\space b_{2}}+b_{1})\\
0& 1 & 0&0&\cdots& 0 & 0& 0& 0& \frac{1}{2}(b_{m}-b_{m-1} + b_{m-2} -b_{m-3} \cdots +b_{5}-b_{4} \underline{+ b_{3}}+b_{2}-b_{1})\\
0 & 0& 1&0&\cdots & 0& 0& 0 & 0& \frac{1}{2}(-b_{m}+b_{m-1} - b_{m-2} +b_{m-3} \cdots -b_{5}\underline{+ b_{4}}+b_{3}-b_{2}+b_{1})\\
0 & 0& 0&1&\cdots & 0& 0& 0 & 0& \frac{1}{2}(b_{m}-b_{m-1} + b_{m-2} -b_{m-3} \cdots \underline{+ b_{5}}+b_{4}-b_{3}+b_{2}-b_{1})\\
\vdots & \vdots & \vdots & \vdots & \ddots& \vdots & \vdots& \vdots &\vdots &\vdots\\
0 & 0& 0&0&\cdots & 1&0& 0& 0& \frac{1}{2}(b_{m}-b_{m-1} \underline{+ b_{m-2}} +b_{m-3} \cdots -b_{5}+b_{4}-b_{3}+b_{2}-b_{1})\\
0 & 0& 0&0&\cdots & 0& 1& 0 &0 &\frac{1}{2}(-b_{m}\underline{+ b_{m-1}} + b_{m-2} -b_{m-3} \cdots +b_{5}-b_{4}+b_{3}-b_{2}+b_{1})\\
0 & 0& 0&0&\cdots & 0& 0& 1&0& \frac{1}{2}(\underline{+ b_{m}}+b_{m-1} - b_{m-2} +b_{m-3} \cdots -b_{5}+b_{4}-b_{3}+b_{2}-b_{1})\\
0 & 0& 0&0&\cdots & 0& 0& 0& 1 & \frac{1}{2}(b_{m}-b_{m-1} + b_{m-2} -b_{m-3} \cdots +b_{5}-b_{4}+b_{3}-b_{2}+b_{1})\\
\end{array}
\right]
\]

$\\*$This tells us that the solution of Equation (\ref{evenoddspan}) is
$$
\alpha_{1} =\frac{1}{2}(-b_{m}+b_{m-1} - b_{m-2} +b_{m-3} \hdots -b_{3}+b_{2}+b_{1})
$$
$$
\alpha_{2} = \frac{1}{2}(b_{m}-b_{m-1} + b_{m-2} -b_{m-3} \hdots +b_{3}+b_{2}-b_{1})
$$
$$
\alpha_{3} =  \frac{1}{2}(-b_{m}+b_{m-1} - b_{m-2} +b_{m-3} \hdots +b_{3}-b_{2}+b_{1})
$$
$$
\alpha_{4} = \frac{1}{2}(b_{m}-b_{m-1} + b_{m-2} -b_{m-3} \hdots -b_{3}+b_{2}-b_{1})
$$
$$
\vdots
$$
$$
\alpha_{m-3} =\frac{1}{2}(b_{m}-b_{m-1} + b_{m-2} +b_{m-3} \hdots -b_{3}+b_{2}-b_{1})
$$
$$
\alpha_{m-2} = \frac{1}{2}(-b_{m}+b_{m-1} + b_{m-2} -b_{m-3} \hdots +b_{3}-b_{2}+b_{1})
$$
$$
\alpha_{m-1} =  \frac{1}{2}(b_{m}+b_{m-1} - b_{m-2} +b_{m-3} \hdots -b_{3}+b_{2}-b_{1})
$$
$$
\alpha_{m} = \frac{1}{2}(b_{m}-b_{m-1} + b_{m-2} -b_{m-3} \hdots +b_{3}-b_{2}+b_{1}).
$$
When $b_{1},\hdots,b_{m}$ are all integers, $(\pm b_{m}\pm b_{m-1} \pm b_{m-2} \pm b_{m-3} \cdots \pm b_{3} \pm b_{2} \pm b_{1})  \equiv (b_{m}+b_{m-1} + b_{m-2} +b_{m-3} \cdots +b_{3}+b_{2}+b_{1}) \Mod{2}$ for any choice of $\pm$. Therefore, when $b_{1},\hdots,b_{m}$ are all integers and $(b_{m}+b_{m-1} + b_{m-2} +b_{m-3} \cdots +b_{3}+b_{2}+b_{1}) \equiv 1 \Mod{2}$ we know ${\alpha_{1}, \hdots , \alpha_{m} }\in \mathbb{Q} \smallsetminus \mathbb{Z} $ and in fact are in $(\frac{1}{2} \mathbb{Z}) \smallsetminus \mathbb{Z}$. Also, when $b_{1},\hdots,b_{m}$ are all integers and $(b_{m}+b_{m-1} + b_{m-2} +b_{m-3} \cdots +b_{3}+b_{2}+b_{1}) \equiv 0 \Mod{2}$ we know ${\alpha_{1}, \hdots , \alpha_{m} }\in \mathbb{Z} $. 

Now consider the map $f: \mathbb{Z}^{m} \xrightarrow{} \mathbb{Z}/2$ defined as $f(\alpha_{1},\alpha_{2},\hdots,\alpha_{m-1},\alpha_{m})= [\alpha_{1}+\alpha_{2}+\cdots+\alpha_{m-1}+\alpha_{m}].$ Note that $\ker(f) = M \cdot \mathbb{Z}^{m}$. Now, $(\im((1-t_{w-1}): \bar{\mathfrak{m}}_{\mathbb{Q}}^{\otimes w}\rightarrow \bar{\mathfrak{m}}_{\mathbb{Q}}^{\otimes w})) \cap \bar{\mathfrak{m}}_{\mathbb{Z}}^{\otimes w}$ is exactly the span over $\mathbb{Z}$ of 
$(x_{k_{1}}\otimes\cdots\otimes x_{k_{m}})^{\otimes \ell}, (x_{k_{m}}\otimes\cdots\otimes x_{k_{m-1}})^{\otimes \ell}, \hdots ,(x_{k_{3}}\otimes\cdots\otimes x_{k_{2}})^{\otimes \ell},$ and $(x_{k_{2}}\otimes\cdots\otimes x_{k_{1}})^{\otimes \ell}$, which are isomorphic to $\mathbb{Z}^{m}$. Under this isomorphism, $\im((1-t_{w-1}):\bar{\mathfrak{m}}_{\mathbb{Z}}^{\otimes w}\rightarrow \bar{\mathfrak{m}}_{\mathbb{Z}}^{\otimes w}) \cong M\cdot \mathbb{Z}^{m}$. Therefore,

$$
\dfrac{\big(\im\big((1-t_{w-1}): \bar{\mathfrak{m}}_{\mathbb{Q}}^{\otimes w}\rightarrow \bar{\mathfrak{m}}_{\mathbb{Q}}^{\otimes w}\big)\big) \cap \bar{\mathfrak{m}}_{\mathbb{Z}}^{\otimes w}}{\im\big((1-t_{w-1}):\bar{\mathfrak{m}}_{\mathbb{Z}}^{\otimes w}\rightarrow \bar{\mathfrak{m}}_{\mathbb{Z}}^{\otimes w}\big)} \cong \dfrac{\mathbb{Z}^{m}}{(M \cdot \mathbb{Z}^{m})} \cong \mathbb{Z}/2.
$$\end{proof}
$\\*$Lemmas \ref{wecl1}, \ref{wocl1},  \ref{woclo},  \ref{wecle}, and \ref{weclo} combine to give:

\begin{lemma}\label{2}
       Consider the family of rings $k[x_1,\hdots,x_d]/\mathfrak{m}^2$, where $\mathfrak{m}$ is the ideal $(x_1,\hdots,x_d)$ and $\bar{\mathfrak{m}} = \mathfrak{m} / \mathfrak{m}^{2}$. For all $w \in \mathbb{Z}^{+}$, 

       $$
        \dfrac{\big(\im\big((1-t_{w-1}): \bar{\mathfrak{m}}_{\mathbb{Q}}^{\otimes w}\rightarrow \bar{\mathfrak{m}}_{\mathbb{Q}}^{\otimes w}\big)\big) \cap \bar{\mathfrak{m}}_{\mathbb{Z}}^{\otimes w}}{\im\big((1-t_{w-1}):\bar{\mathfrak{m}}_{\mathbb{Z}}^{\otimes w}\rightarrow \bar{\mathfrak{m}}_{\mathbb{Z}}^{\otimes w}\big)}\cong \displaystyle\bigoplus_{\substack{m \mid w \\ m \not\equiv w \Mod{2}}}\displaystyle\bigoplus_{\omega_{m,d}} \mathbb{Z}/2
        $$
        
$\\*$where $\omega_{m,d} = \{\textrm{all cycle families of words of length $m$ of cycle length $m$ in $x_1,\hdots,x_d$}\}$. (Note that if $w$ is odd, then there are no $m$ with $m \mid w$ such that $m \not\equiv w \Mod{2}$.)
        
\end{lemma}

\subsection{Calculating $HC_{w-1}^{(w)}$}\hfill

In Proposition \ref{applykertozandq}, we saw that
$$
HC_{w-1}^{(w)}(\mathbb{Z}[x_1,x_2,\hdots,x_d]/\mathfrak{m}^2)\cong \frac{ \bar{\mathfrak{m}}_{\mathbb{Z}}^{\otimes w}}{\im\big((1-t_{w-1}):\bar{\mathfrak{m}}_{\mathbb{Z}}^{\otimes w}\rightarrow \bar{\mathfrak{m}}_{\mathbb{Z}}^{\otimes w}\big)}. $$
\begin{lemma}\label{wecl1.m}
For $w$ even, each word of cycle length $m=1$ will contribute a copy of $\mathbb{Z}/2$ to $\dfrac{ \bar{\mathfrak{m}}_{\mathbb{Z}}^{\otimes w}}{\im\big((1-t_{w-1}):\bar{\mathfrak{m}}_{\mathbb{Z}}^{\otimes w}\rightarrow \bar{\mathfrak{m}}_{\mathbb{Z}}^{\otimes w}\big)}$.
\end{lemma}
\begin{proof}
For $w$ even and $i \in \{1,2,\hdots,d\}$, 
$$
(1-t_{w-1})(x_{i}^{\otimes w})= x_{i}^{\otimes w} + x_{i}^{\otimes w} = 2(x_{i}^{\otimes w}).
$$
\end{proof}

\begin{lemma}\label{wocl1.m}
For $w$ odd, each word of cycle length $m=1$ will contribute a copy of $\mathbb{Z}$ to \newline $\dfrac{ \bar{\mathfrak{m}}_{\mathbb{Z}}^{\otimes w}}{\im\big((1-t_{w-1}):\bar{\mathfrak{m}}_{\mathbb{Z}}^{\otimes w}\rightarrow \bar{\mathfrak{m}}_{\mathbb{Z}}^{\otimes w}\big)}$.
\end{lemma}
\begin{proof}
For $w$ odd and $i \in \{1,2,\hdots,d\}$, 
$$
(1-t_{w-1})(x_{i}^{\otimes w})= x_{i}^{\otimes w} - x_{i}^{\otimes w} = 0.
$$
\end{proof}

\begin{lemma}\label{woclo.m}

For any $w>1$ odd, each cycle family of words of cycle length $m>1$ will contribute a copy of $\mathbb{Z}$ to $\dfrac{\bar{\mathfrak{m}}_{\mathbb{Z}}^{\otimes w}}{\im\big((1-t_{w-1}):\bar{\mathfrak{m}}_{\mathbb{Z}}^{\otimes w}\rightarrow \bar{\mathfrak{m}}_{\mathbb{Z}}^{\otimes w}\big)}$.
\end{lemma}
\begin{proof} $w=m\cdot \ell$. As before, every tensor monomial $(x_{k_{1}}\otimes\cdots\otimes x_{k_{m}})^{\otimes \ell}$ of cycle length $m$ has $m$ other words of length $w$ in its cycle family. Quotienting out by the image of $(1-t_{w-1})$ identifies these $m$ generators with each other, to give the unique generator of 
$\frac{\mathbb{Z}^{\oplus m}}{im(1-t_{w-1})} \cong \mathbb{Z}$.
\end{proof}
\begin{lemma}\label{wecle.m}
For any $w>1 $ even, each cycle family of words of cycle length $m$ also even will contribute a copy of $\mathbb{Z}$ to $\dfrac{\bar{\mathfrak{m}}_{\mathbb{Z}}^{\otimes w}}{\im\big((1-t_{w-1}):\bar{\mathfrak{m}}_{\mathbb{Z}}^{\otimes w}\rightarrow \bar{\mathfrak{m}}_{\mathbb{Z}}^{\otimes w}\big)}$.
\end{lemma}
\begin{proof} Again, there are $m$ tensor monomials of length $w$ in the cycle family of this word. Quotienting out by the image of $(1-t_{w-1})$ identifies the following generators
$$
(x_{k_{1}}\otimes\cdots\otimes x_{k_{m}})^{\otimes \ell}
\sim
-(x_{k_{m}}\otimes\cdots\otimes x_{k_{m-1}})^{\otimes \ell}
\sim
\cdots
\sim
(x_{k_{3}}\otimes\cdots\otimes x_{k_{2}})^{\otimes \ell}
\sim
-(x_{k_{2}}\otimes\cdots\otimes x_{k_{1}})^{\otimes \ell}
$$
so again we are just left with a single copy of $\mathbb{Z}$.
\end{proof}
\begin{lemma}\label{weclo.m}
For any $w>2$ even, each cycle family of words of cycle length $m>1$ with $m$ odd will contribute a copy of $\mathbb{Z}/2$ to $\dfrac{\bar{\mathfrak{m}}_{\mathbb{Z}}^{\otimes w}}{\im\big((1-t_{w-1}):\bar{\mathfrak{m}}_{\mathbb{Z}}^{\otimes w}\rightarrow \bar{\mathfrak{m}}_{\mathbb{Z}}^{\otimes w}\big)}$.
\end{lemma}
\begin{proof} Again, there are $m$ monomials of length $w$ in the cycle family of this word. When we quotient the $\mathbb{Z}^{\oplus m}$ generated by the tensor monomials in the cycle family of $(x_{k_{1}}\otimes\cdots\otimes x_{k_{m}})^{\otimes \ell}$ by the image of $(1-t_{w-1})$, we get
$$
(x_{k_{1}}\otimes\cdots\otimes x_{k_{m}})^{\otimes \ell}
\sim
-(x_{k_{m}}\otimes\cdots\otimes x_{k_{m-1}})^{\otimes \ell}
\sim
\cdots
\sim
(x_{k_{2}}\otimes\cdots\otimes x_{k_{1}})^{\otimes \ell}
\sim
-(x_{k_{1}}\otimes\cdots\otimes x_{k_{m}})^{\otimes \ell}
$$
so we are just left with $\mathbb{Z}/2$ generated by any one of the elements.
\end{proof}
Gathering the results in Lemmas \ref{wecl1.m}, \ref{wocl1.m}, \ref{woclo.m},  \ref{wecle.m}, and \ref{weclo.m} give us the following Lemma.
\begin{lemma}\label{3}
       Consider the family of rings $k[x_1,\hdots,x_d]/\mathfrak{m}^2$, where $\mathfrak{m}$ is the ideal $(x_1,\hdots,x_d)$ and $\bar{\mathfrak{m}} = \mathfrak{m} / \mathfrak{m}^{2}$. For all $w \in \mathbb{Z}^{+}$, 
               $$\dfrac{\bar{\mathfrak{m}}_{\mathbb{Z}}^{\otimes w}}{\im((1-t_{w-1}):\bar{\mathfrak{m}}_{\mathbb{Z}}^{\otimes w}\rightarrow \bar{\mathfrak{m}}_{\mathbb{Z}}^{\otimes w})} \cong         \Bigg(\displaystyle\bigoplus_{\substack{m \mid w \\ m \equiv w \Mod{2}}}\displaystyle\bigoplus_{\omega_{m,d}} \mathbb{Z}\Bigg) \bigoplus \Bigg(        \displaystyle\bigoplus_{\substack{m \mid w \\ m \not\equiv w \Mod{2}}}\displaystyle\bigoplus_{\omega_{m,d}} \mathbb{Z}/2\Bigg)$$
where $\omega_{m,d} = \{\textrm{all cycle families of words of length $m$ of cycle length $m$ in $x_1,\hdots,x_d$}\}$. (Note that if $w$ is odd, then there are no $m$ with $m \mid w$ such that $m \not\equiv w \Mod{2}$.)
\end{lemma}
Proposition \ref{applykertozandq}, Lemma \ref{1}, Lemma \ref{2}, and  Lemma \ref{3} give

\begin{theorem}\label{Zcyclic}
Let $A = \mathbb{Z}[x_1,x_2,\hdots,x_d]/\mathfrak{m}^2$ where $\mathfrak{m}$ is the ideal $(x_1,x_2,\hdots,x_d)$ and $\bar{\mathfrak{m}} = \mathfrak{m} / \mathfrak{m}^{2}$. Then
\begin{center}
$
HC_{n}^{(0)}(A) \cong 
    \begin{dcases}
        \mathbb{Z} & n  \textrm{ even and } n \geq 0 \\
        0 & n \textrm{ else} \\
    \end{dcases}
$
\end{center}
and for $w > 0$
\begin{center}
$
HC_{n}^{(w)}(A) \cong 
    \begin{dcases} 
       \>\>\>\>\>\>\>\>\>\>\>\>\>\>\>\>\>\>\>\>\>\>\>\>\> 0  & n \leq w-2 \\
       \\
        \Bigg(\displaystyle\bigoplus_{\substack{m \mid w \\ m \equiv w \Mod{2}}}\displaystyle\bigoplus_{\omega_{m,d}} \mathbb{Z}\Bigg) \bigoplus \Bigg(        \displaystyle\bigoplus_{\substack{m \mid w \\ m \not\equiv w \Mod{2}}}\displaystyle\bigoplus_{\omega_{m,d}} \mathbb{Z}/2\Bigg) & n = w-1 \\
       \\
       \displaystyle\bigoplus_{\substack{m \mid w \\ m \equiv w \Mod{2}}}\displaystyle\bigoplus_{\omega_{m,d}} \mathbb{Z}/(\tfrac{w}{m})
         & n = w + 2i,\hspace{11.5mm} i\geq 0 \\
        \\
       \displaystyle\bigoplus_{\substack{m \mid w \\ m \not\equiv w \Mod{2}}}\displaystyle\bigoplus_{\omega_{m,d}} \mathbb{Z}/2 & n = w + 1 + 2i, \hspace{5mm}i\geq0 \>  \\
    \end{dcases}
$
\end{center}
where $\omega_{m,d} = \{\textrm{all cycle families of words length $m$ and cycle length $m$ in $x_1,\hdots,x_d$}\}$  has order \newline $\frac{1}{m}\sum_{i \vert m}\mu(m/i) d^{i}$.
\end{theorem}

\section{Negative Cyclic Homology of $k[x_1,x_2,\hdots,x_d]/\mathfrak{m}^2$}\hfill

We now want to calculate $HC_{*}^{-}(A)$ for $A = k[x_1,x_2,\hdots,x_d]/\mathfrak{m}^2$, to do this we will compute the total homology of the double complex in Definition \ref{negative}. Again, we can study the resulting calculation of $HC_{*}^{-}(A)$ by splitting it into $(HC^{-})_{*}^{(w)}(A)$ for all  $w \geq 0 $ like we did for $HC_{*}^{(w)}(A)$. 

\begin{center}
    $HC_{*}^{-}(A) \cong \displaystyle\bigoplus_{w=0}^{\infty}(HC^{-})_{*}^{(w)}(A)$
\end{center}
As in Equation (\ref{HCW0}), if $w=0$ the $E^{1}_{p,q} = 0$ unless $p\leq 0$ is even and $q=0$, so $E^{1} \cong E^{\infty}$ and
 
\begin{equation}\label{HCNW0}
(HC^{-})_{n}^{(0)}(A) \cong 
    \begin{dcases}
        k  & n \textrm{ even and } n \leq 0 \\
        0 &  \textrm{else} \\
    \end{dcases}.
\end{equation}
Again for $w>0$ in the $E^{1}$-page, the even columns (numbered less than or equal to 0) consist of $HH_{*}(A)$ and the odd columns are zero. So we again have  $\partial^{1}=0$ and $E^{1} \cong E^{2}$. Again, all $\partial^{r}$ for $r\geq 3$ are $0$ for dimension reasons. Therefore, the $E^{3}$ page is same as the $E^{\infty}$ page, and we get
 
\begin{equation}\label{HNCNW}
(HC^{-})_{n}^{(w)}(A) \cong 
\begin{dcases}
0  & n > w \\
\ker(1-t_{w-1}) & n = w \\
\ker(\partial^{2}) & n = w -1-2i,\hspace{5mm} i\geq 0  \\
\coker(\partial^{2}) & n = w - 2i,\hspace{11.5mm} i> 0.\\
\end{dcases}
\end{equation}

Following the same reasoning that appears after Equation (\ref{HCNW}), we can rewrite this as

\begin{equation}\label{HNCNW2}
(HC^{-})_{n}^{(w)}(A) \cong 
\begin{dcases}
0  & n > w \\
\ker(1-t_{w-1}) & n = w \\
 \ker(\coker(1-t_{w-1}) \xrightarrow{N_{w-1}} \ker(1-t_{w-1})) & n = w -1-2i,\hspace{5mm} i\geq 0 \\
\coker(\coker(1-t_{w-1}) \xrightarrow{N_{w-1}} \ker(1-t_{w-1})) & n = w - 2i,\hspace{11.5mm} i> 0. \\
\end{dcases}
\end{equation}

\medskip

\subsection{Comparison to Ext}\hfill

The calculation of the pieces of Equation (\ref{HNCNW2}) is exactly what we get computing $Ext_{n}^{k[C_{w}]}(k,\bar{\mathfrak{m}}^{\otimes w})$ where $C_{w}= \langle \alpha : \alpha^{(w)}=1  \rangle$ is the cyclic group of order $w$ and $\bar{\mathfrak{m}}^{\otimes w}$ is a $k[C_{w}]$-module with the action
$$
\alpha(x_{n_{1}}\otimes\cdots\otimes x_{n_{w}}) = t_{w-1}(x_{n_{1}}\otimes\cdots\otimes x_{n_{w}}) = (-1)^{w-1}(x_{n_{w}}\otimes x_{n_{1}} \otimes\cdots\otimes x_{n_{-1}}).
$$
To compute $Ext_{n}^{k[C_{w}]}(k,\bar{\mathfrak{m}}^{\otimes w})$, we apply  $\Hom_{k[C_{w}]}(-,\bar{\mathfrak{m}}^{\otimes w})$ to the resolution in Equation (\ref{r}) to get  
\begin{center}
\begin{tikzcd}[
  ar symbol/.style = {draw=none,"#1" description,sloped},
  isomorphic/.style = {ar symbol={\cong}},
  equals/.style = {ar symbol={=}},
  ]
\cdots& \Hom(k[C_{w}],\bar{\mathfrak{m}}^{\otimes w}) \arrow{l}{(1-\alpha) \circ \_}& \Hom(k[C_{w}],\bar{\mathfrak{m}}^{\otimes w})\arrow{l}{N \circ \_}&\Hom(k[C_{w}],\bar{\mathfrak{m}}^{\otimes w})\arrow{l}{(1-\alpha) \circ \_}& 0\arrow{l} \> 
\end{tikzcd}
\end{center}
so $Ext_{n}^{k[C_{w}]}(k,\bar{\mathfrak{m}}^{\otimes w})$ is the cohomology of the complex
\begin{equation}\label{cohomExt}
\cdots\xleftarrow{1-t_{w-1}} \bar{\mathfrak{m}}^{\otimes w}\xleftarrow{N_{w-1}} \bar{\mathfrak{m}}^{\otimes w}\xleftarrow{1-t_{w-1}} \bar{\mathfrak{m}}^{\otimes w}\xleftarrow{N_{w-1}} \bar{\mathfrak{m}}^{\otimes w}\xleftarrow{1-t_{w-1}}\bar{\mathfrak{m}}^{\otimes w}\xleftarrow{} 0
\end{equation}\label{complex}
$\\*$yielding

\begin{equation}\label{ext0}
Ext_{0}^{k[C_{w}]}(k,\bar{\mathfrak{m}}^{\otimes w})\cong  \textrm{ker}(1-t_{w-1})
\end{equation}
\begin{equation}\label{ext1}
Ext_{1+ 2i}^{k[C_{w}]}(k,\bar{\mathfrak{m}}^{\otimes w})\cong\ker(\textrm{coker}(1-t_{w-1}) \xrightarrow{N_{w-1}} \ker(1-t_{w-1})) \textrm{ for }  i\geq 0 \> 
\end{equation}
\begin{equation}\label{ext2}
Ext_{2 + 2i}^{k[C_{w}]}(k,\bar{\mathfrak{m}}^{\otimes w})\cong\textrm{coker}(\textrm{coker}(1-t_{w-1}) \xrightarrow{N_{w-1}} \ker(1-t_{w-1})) \textrm{ for }  i\geq 0.\>
\end{equation}
We deduce
\begin{theorem}\label{HNCNW3}
Let $A = k[x_1,x_2,\hdots,x_d]/\mathfrak{m}^2$, where $k$ is any commutative unital ring, $\mathfrak{m}$ is the ideal $(x_1,x_2,\hdots,x_d)$, $\bar{\mathfrak{m}} = \mathfrak{m} / \mathfrak{m}^{2}$, $C_{w}= \langle \alpha  : \alpha^{(w)}=1  \rangle$, and $\bar{\mathfrak{m}}^{\otimes w}$ is a $k[C_{w}]$-module with the action
$$
\alpha(x_{n_{1}}\otimes\cdots\otimes x_{n_{w}}) = (-1)^{w-1}(x_{n_{w}}\otimes x_{n_{1}} \otimes\cdots\otimes x_{n_{w-1}}) \> .
$$
$\\*$ Then for $w = 0$

\begin{center}
$
(HC^{-})_{n}^{(0)}(A) \cong 
    \begin{dcases}
        k  & n \textrm{ even and } n \leq 0 \\
        0 & n \textrm{ else} \\
    \end{dcases}
$
\end{center}

\medskip

$\\*$ and for $w > 0$

\begin{center}
$
(HC^{-})_{n}^{(w)}(A) \cong 
    \begin{dcases} 
        0  & n > w \\
       Ext_{w-n}^{k[C_{w}]}(k,\bar{\mathfrak{m}}^{\otimes w}) & n \leq w\>  \\
    \end{dcases}
$
\end{center}
$\\*$ which can be rewritten as 

\begin{center}
$
(HC^{-})_{n}^{(w)}(A) \cong 
    \begin{dcases} 
        0  & n > w \\
       H^{w-n}(C_{w};\bar{\mathfrak{m}}^{\otimes w}) & n \leq w\>  \\
    \end{dcases}
$
\end{center}
$\\*$for the $C_{w}$ action on $\bar{\mathfrak{m}}^{\otimes w}$ explained above.

\end{theorem}

\subsection{Negative Cyclic Homology of $\mathbb{Q}[x_1,x_2,\hdots,x_d]/\mathfrak{m}^2$}\hfill

Since $\mathbb{Q}$ is a projective $\mathbb{Q}[C_{w}]$-module, we get that for $A = \mathbb{Q}[x_1,x_2,\hdots,x_d]/\mathfrak{m}^2$ and for $w = 0$

\begin{equation}
(HC^{-})_{n}^{(0)}(A) \cong 
    \begin{dcases}
        \mathbb{Q}  & n \textrm{ even and } n \leq 0  \\
        0 & \textrm{else} \\
    \end{dcases}
\end{equation}

\medskip

$\\*$ and for $w > 0$

\begin{equation}
(HC^{-})_{n}^{(w)}(A) \cong 
    \begin{dcases} 
       Hom_{\mathbb{Q}[C_{w}]}(\mathbb{Q},\bar{\mathfrak{m}}^{\otimes w})  & n = w \\
        0  & \textrm{else} \>. 
    \end{dcases}
\end{equation}

Since $Hom_{\mathbb{Q}[C_{w}]}(\mathbb{Q},\bar{\mathfrak{m}}^{\otimes w})$ is the set of $\mathbb{Q}[C_{w}]$-module homomorphisms $f:\mathbb{Q} \rightarrow \bar{\mathfrak{m}}^{\otimes w}$, we know any $f$ must satisfy 
$f(\alpha  q ) = \alpha f(q)$ for all $q \in \mathbb{Q}$, where $\alpha$ acts trivially on $q$ and
$\alpha(x_{n_{1}}\otimes\cdots\otimes x_{n_{w}}) = (-1)^{w-1}(x_{n_{w}}\otimes x_{n_{1}} \otimes\cdots\otimes x_{n_{w-1}})$ on $f(q)$. We can rewrite the action $\alpha f(q)$ as $t_{w-1}f(q)$. We then get $f(q) = t_{w-1} f(q)$ which can be rewritten as $(1-t_{w-1})f(q) = 0$. Therefore the set of possible $\mathbb{Q}[C_{w}]$-module homomorphisms $f:\mathbb{Q} \rightarrow \bar{\mathfrak{m}}^{\otimes w}$ are exactly in one to one correspondence with a choice of $f(1) \in \ker(1-t_{w-1})$, and
               $$\ker\big((1-t_{w-1}):\bar{\mathfrak{m}}_{\mathbb{Q}}^{\otimes w}\rightarrow \bar{\mathfrak{m}}_{\mathbb{Q}}^{\otimes w}\big) \cong         \displaystyle\bigoplus_{\substack{m \mid w \\ m \equiv w \Mod{2}}}\displaystyle\bigoplus_{\omega_{m,d}} \mathbb{Q}$$
by the same reasoning that will be used to prove Lemma \ref{ker1minustwforz}. We deduce

\begin{corollary}
Let $A = \mathbb{Q}[x_1,...,x_d]/\mathfrak{m}^2$ where $\mathfrak{m}$ is the ideal $(x_1,\hdots,x_d)$. Then for $w = 0$
\begin{center}
$
(HC^{-})_{n}^{(0)}(A) \cong 
    \begin{dcases}
        \mathbb{Q}  & n \textrm{ even and } n \leq 0  \\
        0 & \textrm{else} \\
    \end{dcases}
$
\end{center}
$\\*$ and for $w > 0$

\begin{center}
$
(HC^{-})_{n}^{(w)}(A) \cong 
    \begin{dcases} 
       \displaystyle\bigoplus_{\substack{m \mid w \\ m \equiv w \Mod{2}}}\displaystyle\bigoplus_{\omega_{m,d}} \mathbb{Q} & n = w \\
        0  & \textrm{else} \>
    \end{dcases}
$
\end{center}
where $\omega_{m,d} = \{\textrm{all cycle families of words length $m$ and cycle length $m$ in $x_1,\hdots,x_d$}\}$  has order \newline $\frac{1}{m}\sum_{i \vert m}\mu(m/i) d^{i}$.
\end{corollary}

\subsection{Negative Cyclic Homology of $\mathbb{Z}[x_1,x_2,\hdots,x_d]/\mathfrak{m}^2$}\hfill

To compute the negative cyclic homology of $\mathbb{Z}[x_1,x_2,\hdots,x_d]/\mathfrak{m}^2$, note the similarity between the complex in Equation (\ref{cohomExt}) and the complex in Equation (\ref{TORHomComplex}). In Section 5, we showed that

\begin{equation}\label{rel1}
    \ker\big(\textrm{coker}(1-t_{w-1}) \xrightarrow{N_{w-1}} \ker(1-t_{w-1})\big) \cong \frac{\big(\im\big((1-t_{w-1}): \bar{\mathfrak{m}}_{\mathbb{Q}}^{\otimes w}\rightarrow \bar{\mathfrak{m}}_{\mathbb{Q}}^{\otimes w}\big)\big) \cap \bar{\mathfrak{m}}_{\mathbb{Z}}^{\otimes w}}{\im\big((1-t_{w-1}):\bar{\mathfrak{m}}_{\mathbb{Z}}^{\otimes w}\rightarrow \bar{\mathfrak{m}}_{\mathbb{Z}}^{\otimes w}\big)}
\end{equation}
\begin{equation}\label{rel2}
    \textrm{coker}\big(\textrm{coker}(1-t_{w-1}) \xrightarrow{N_{w-1}} \ker(1-t_{w-1})\big)\cong         \frac{\big(\im(N_{w-1}: \bar{\mathfrak{m}}_{\mathbb{Q}}^{\otimes w}\rightarrow \bar{\mathfrak{m}}_{\mathbb{Q}}^{\otimes w})\big) \cap \bar{\mathfrak{m}}_{\mathbb{Z}}^{\otimes w}}{\im(N_{w-1}:\bar{\mathfrak{m}}_{\mathbb{Z}}^{\otimes w}\rightarrow \bar{\mathfrak{m}}_{\mathbb{Z}}^{\otimes w})}.
\end{equation}

Applying Theorem \ref{HNCNW3} we get
\begin{theorem}\label{applykertozandq2}
Let $A = \mathbb{Z}[x_1,x_2,\hdots,x_d]/\mathfrak{m}^2$, where $\mathfrak{m}$ is the ideal $(x_1,x_2,\hdots,x_d)$ and $\bar{\mathfrak{m}} = \mathfrak{m} / \mathfrak{m}^{2}$.
Then
\begin{center}
$
(HC^{-})_{n}^{(0)}(A) \cong 
    \begin{dcases}
        \mathbb{Z} & n \textrm{ even and } n \leq 0 \\
        0 & \textrm{else}\\
    \end{dcases}
$
\end{center}
and for $w > 0$
\begin{center}
$
(HC^{-})_{n}^{(w)}(A) \cong 
    \begin{dcases} 
       \>\>\>\>\>\>\>\>\>\>\>\>\>\>\>\>\>\>\>\>\>\>\>\>\> 0  & n > w \\
       \\
       \ker\big((1-t_{w-1}):\bar{\mathfrak{m}}_{\mathbb{Z}}^{\otimes w}\rightarrow \bar{\mathfrak{m}}_{\mathbb{Z}}^{\otimes w}\big) & n = w \\
       \\
        \frac{\big(\im(N_{w-1}: \bar{\mathfrak{m}}_{\mathbb{Q}}^{\otimes w}\rightarrow \bar{\mathfrak{m}}_{\mathbb{Q}}^{\otimes w})\big) \cap \bar{\mathfrak{m}}_{\mathbb{Z}}^{\otimes w}}{\im(N_{w-1}:\bar{\mathfrak{m}}_{\mathbb{Z}}^{\otimes w}\rightarrow \bar{\mathfrak{m}}_{\mathbb{Z}}^{\otimes w})} & n = w - 1 - 2i, \hspace{5mm} i\geq 0 \\
        \\
       \frac{\big(\im\big((1-t_{w-1}): \bar{\mathfrak{m}}_{\mathbb{Q}}^{\otimes w}\rightarrow \bar{\mathfrak{m}}_{\mathbb{Q}}^{\otimes w}\big)\big) \cap \bar{\mathfrak{m}}_{\mathbb{Z}}^{\otimes w}}{\im\big((1-t_{w-1}):\bar{\mathfrak{m}}_{\mathbb{Z}}^{\otimes w}\rightarrow \bar{\mathfrak{m}}_{\mathbb{Z}}^{\otimes w}\big)} & n = w - 2i, \hspace{11.5mm} i> 0 \> . \\
    \end{dcases}
$

\end{center}

\end{theorem}
$\\*$We can use Lemma \ref{1} and Lemma \ref{2} to understand the last two parts of the above equation for $(HC^{-})_{n}^{(w)}(A)$. However, we still need to compute $\ker\big((1-t_{w-1}):\bar{\mathfrak{m}}_{\mathbb{Z}}^{\otimes w}\rightarrow \bar{\mathfrak{m}}_{\mathbb{Z}}^{\otimes w}\big)$. 

\begin{lemma}\label{wel1nc}
For $w$ even, each cycle family of words of cycle length $m=1$ will contribute nothing to $\ker\big((1-t_{w-1}):\bar{\mathfrak{m}}_{\mathbb{Z}}^{\otimes w}\rightarrow \bar{\mathfrak{m}}_{\mathbb{Z}}^{\otimes w}\big)$.
\end{lemma}
\begin{proof}
For $w$ even and $i \in \{1,2,\hdots,d\}$, 
$$
(1-t_{w-1})(x_{i}^{\otimes w})= x_{i}^{\otimes w} + x_{i}^{\otimes w} = 2(x_{i}^{\otimes w}).
$$
\end{proof}
\begin{lemma}\label{wol1nc}
For $w$ odd, each cycle family of words of cycle length $m=1$ will contribute a copy of $\mathbb{Z}$ to $\ker\big((1-t_{w-1}):\bar{\mathfrak{m}}_{\mathbb{Z}}^{\otimes w}\rightarrow \bar{\mathfrak{m}}_{\mathbb{Z}}^{\otimes w}\big)$.
\end{lemma}
\begin{proof}
For $w$ odd and $i \in \{1,2,\hdots,d\}$, 
$$
(1-t_{w-1})(x_{i}^{\otimes w})= x_{i}^{\otimes w} - x_{i}^{\otimes w} = 0.
$$
\end{proof}
\begin{lemma}\label{wolonc}
For $w>1$ odd, each cycle family of words of cycle length $m>1$ will contribute a copy of $\mathbb{Z}$ to $\ker\big((1-t_{w-1}):\bar{\mathfrak{m}}_{\mathbb{Z}}^{\otimes w}\rightarrow \bar{\mathfrak{m}}_{\mathbb{Z}}^{\otimes w}\big)$.
\end{lemma}
\begin{proof}
First note, if $m|w$ and $w$ is odd, $m$ is also odd. There exists an $\ell$ such that $w=m\cdot \ell$. Again, there are $m$ monomials of length $w$ in the cycle family of this word. Consider the image of these words under the map: $( 1-t_{w-1}): \bar{\mathfrak{m}}_{\mathbb{Z}}^{\otimes w}\rightarrow \bar{\mathfrak{m}}_{\mathbb{Z}}^{\otimes w}$. The only linear combinations of words in this cycle family that will go to $0$ under the map $1-t_{w-1}$ must be a multiple of the sum of all the words in this cycle family. 
\end{proof}
\begin{lemma}\label{welenc}
For $w>1$ even, each cycle family of words of cycle length $m$ also even will contribute a copy of $\mathbb{Z}$ to $\ker\big((1-t_{w-1}):\bar{\mathfrak{m}}_{\mathbb{Z}}^{\otimes w}\rightarrow \bar{\mathfrak{m}}_{\mathbb{Z}}^{\otimes w}\big)$.
\end{lemma}
\begin{proof}
Again, there are $m$ monomials of length $w$ in the cycle family of this word. Consider the image of these words under the map: $\big( 1-t_{w-1} \big): \bar{\mathfrak{m}}_{\mathbb{Z}}^{\otimes w}\rightarrow \bar{\mathfrak{m}}_{\mathbb{Z}}^{\otimes w}$. Note the following:
$$(1-t_{w-1})\big((x_{k_{1}}\otimes\cdots\otimes x_{k_{m}})^{\otimes \ell}-(x_{k_{m}}\otimes\cdots\otimes x_{k_{m-1}})^{\otimes \ell} + (x_{k_{m-1}}\otimes\cdots\otimes x_{k_{m-2}})^{\otimes \ell}- \ldots
$$
$$+(x_{k_{3}}\otimes\cdots\otimes x_{k_{2}})^{\otimes \ell} -(x_{k_{2}}\otimes\cdots\otimes x_{k_{1}})^{\otimes \ell}\big)=0
$$
$\\*$The only linear combinations of words in this cycle family that will go to $0$ under the map $1-t_{w-1}$ must be a multiple of the alternating sum (in the above order) of all the words in this cycle family which works in this case particularly because $m$ is even and there are $m$ words in this cycle family. Therefore, any element of $\ker(1-t_{w-1})$ that comes from this cycle family must be generated by the alternating sum (in the above order) of all members of this cycle family.
\end{proof}
\begin{lemma}\label{welonc}
For $w>1$ even, each cycle family of words of cycle length $m>1$ with $m$ odd will contribute nothing to $\ker\big((1-t_{w-1}):\bar{\mathfrak{m}}_{\mathbb{Z}}^{\otimes w}\rightarrow \bar{\mathfrak{m}}_{\mathbb{Z}}^{\otimes w}\big)$.
\end{lemma}
\begin{proof}
This is almost identical to the proof above of Lemma \ref{welenc}, except here the alternating sum of all the words of this cycle family will not go to $0$ because there are $m$ (which is odd) words in this cycle family. In fact, there is only the trivial linear combination of the above words that will be in the $\ker((1-t_{w-1})$. So the words in this cycle family will contribute nothing to $\ker((1-t_{w-1})$.
\end{proof}
Lemmas \ref{wel1nc}, \ref{wol1nc}, \ref{wolonc}, \ref{welenc}, and \ref{welonc} give the following lemma.
\begin{lemma}\label{ker1minustwforz}
               $$\ker\big((1-t_{w-1}):\bar{\mathfrak{m}}_{\mathbb{Z}}^{\otimes w}\rightarrow \bar{\mathfrak{m}}_{\mathbb{Z}}^{\otimes w}\big) \cong         \displaystyle\bigoplus_{\substack{m \mid w \\ m \equiv w \Mod{2}}}\displaystyle\bigoplus_{\omega_{m,d}} \mathbb{Z}$$
\end{lemma}
Summarizing Lemma \ref{1}, Lemma \ref{2}, Lemma \ref{ker1minustwforz}, and Theorem \ref{applykertozandq2}, we get 
\begin{theorem}\label{NegCycZ}

Let $A = \mathbb{Z}[x_1,x_2,\hdots,x_d]/\mathfrak{m}^2$, where $\mathfrak{m}$ is the ideal $(x_1,x_2,\hdots,x_d)$ and $\bar{\mathfrak{m}} = \mathfrak{m} / \mathfrak{m}^{2}$.
Then
\begin{center}
$
(HC^{-})_{n}^{(0)}(A) \cong 
    \begin{dcases}
        \mathbb{Z} & n \textrm{ even and } n \leq 0 \\
        0 & \textrm{else}\\
    \end{dcases}
$
\end{center}
and for $w > 0$

\begin{center}
$
(HC^{-})_{n}^{(w)}(A) \cong 
    \begin{dcases} 
       \>\>\>\>\>\>\>\>\>\>\>\>\>\>\>\>\>\>\>\>\>\>\>\>\> 0  & n > w \\
       \\
       \displaystyle\bigoplus_{\substack{m \mid w \\ m \equiv w \Mod{2}}}\displaystyle\bigoplus_{\omega_{m,d}} \mathbb{Z} & n = w \\
        \\
         \displaystyle\bigoplus_{\substack{m \mid w \\ m \not\equiv w \Mod{2}}}\displaystyle\bigoplus_{\omega_{m,d}} \mathbb{Z}/2
         & n = w -1 - 2i,\hspace{5mm} i\geq 0 \\
        \\
      \displaystyle\bigoplus_{\substack{m \mid w \\ m \equiv w \Mod{2}}}\displaystyle\bigoplus_{\omega_{m,d}} \mathbb{Z}/\big(\tfrac{w}{m}\big)
       & n = w -2i,\hspace{11.5mm} i> 0  \\
    \end{dcases}
$

\end{center}
$\\*$where $\omega_{m,d} = \{\textrm{all cycle families of words length $m$ and cycle length $m$ in $x_1,\hdots,x_d$}\}$ has order \newline $\frac{1}{m}\sum_{i \vert m}\mu(m/i) d^{i}$.
\end{theorem}

\section{Periodic Cyclic Homology of $k[x_1,x_2,\hdots,x_d]/\mathfrak{m}^2$}
We now want to calculate $HP_{*}(A)$ for $A = k[x_1,x_2,\hdots,x_d]/\mathfrak{m}^2$ by computing the total homology of the double complex in Definition \ref{periodic}. As before, we break the calculation down by weight,

\begin{center}
    $HP_{*}(A) \cong \displaystyle\bigoplus_{w=0}^{\infty}HP_{*}^{(w)}(A)$
\end{center}

$\\*$and as before, the case $w=0$ is easy yielding
\begin{equation}
HP_{n}^{(0)}(A) \cong 
    \begin{dcases}
        k  & n \textrm{ even} \\
        0 &  \textrm{else} \\
    \end{dcases}.
\end{equation}

$\\*$For $w>0$, we will again get $E^{1}\cong E^{2}$ and $E^{3}\cong E^{\infty}$ so
\begin{equation}\label{HPNW}
HP_{n}^{(w)}(A) \cong 
    \begin{dcases}
        \textrm{coker}(\partial^{2}) & n = w + 2i  \textrm{ for } i\in\mathbb{Z}\\
        \ker(\partial^{2}) & n = w + 2i +1  \textrm{ for } i\in\mathbb{Z} \> . \\
    \end{dcases}
\end{equation}

$\\*$Following the same reasoning that appears after Equation (\ref{HCNW}), we can this as

\begin{equation}\label{HPNW2}
   HP_{n}^{(w)}(A) \cong 
    \begin{dcases}
        \textrm{coker}\big(\textrm{coker}(1-t_{w-1}) \xrightarrow{N_{w-1}} \ker(1-t_{w-1})\big) & n = w + 2i  \textrm{ for } i\in\mathbb{Z}\\
        \ker\big(\textrm{coker}(1-t_{w-1}) \xrightarrow{N_{w-1}} \ker(1-t_{w-1})\big) & n = w + 2i +1  \textrm{ for } i\in\mathbb{Z} \> . \\
    \end{dcases}
\end{equation}
\begin{definition}
The \textbf{Tate cohomology groups} $\hat{H}^{n}(G,A)$ of a discrete group $G$ with coefficients in a $\mathbb{Z}[G]$-module $A$ are defined to be 
$$
   \hat{H}^{n}(G,A) = 
    \begin{dcases}
        H^{n}(G,A) & n \geq 1\\
        \coker N & n = 0\\
        \ker N & n = -1\\
        H_{-(n+1)}(G,A) & n \leq -2\\
    \end{dcases}
$$
\end{definition}

Comparing Equation (\ref{HPNW2}) with Equations (\ref{HCNW2}) and (\ref{HNCNW2}), and then comparing the group homology and cohomology results from Theorems \ref{HCNW3} and \ref{HNCNW3} with the above definition of Tate cohomology, we get the following corollary.
\medskip
\begin{corollary}\label{HPNW3}
For $A = k[x_1,x_2,\hdots,x_d]/\mathfrak{m}^2$, where $k$ is any commutative unital ring, $\mathfrak{m}$ is the ideal $(x_1,x_2,\hdots,x_d)$, $\bar{\mathfrak{m}} = \mathfrak{m} / \mathfrak{m}^{2}$, $HP_{*}(A) \cong \displaystyle\bigoplus_{w=0}^{\infty}HP_{*}^{(w)}(A)$, where $HP_{*}^{(w)}(A)$ is the homology of the subcomplex of weight $w$ terms, and we let $C_{w}= \langle \alpha : \alpha^{(w)}=1 \rangle$ acts on  $\bar{\mathfrak{m}}^{\otimes w}$  by 
$$
\alpha(x_{n_{1}}\otimes\cdots\otimes x_{n_{w}}) = (-1)^{w-1}(x_{n_{w}}\otimes x_{n_{1}} \otimes\cdots\otimes x_{n_{w-1}}) \> .
$$
we get for $w = 0$

\begin{center}
$
HP_{n}^{(0)}(A) \cong 
    \begin{dcases}
        k  & n \textrm{ even} \\
        0 & n \textrm{ odd} \\
    \end{dcases}
$
\end{center}

\medskip

$\\*$ and for $w > 0$

\begin{center}
$
   HP_{n}^{(w)}(A) = 
    \begin{dcases}
        \coker\big(\coker(1-t_{w-1}) \xrightarrow{N_{w-1}} \ker(1-t_{w-1})\big) & n = w + 2i  \textrm{ for } i\in\mathbb{Z}\\
        \ker\big(\coker(1-t_{w-1}) \xrightarrow{N_{w-1}} \ker(1-t_{w-1})\big) & n = w + 2i +1  \textrm{ for } i\in\mathbb{Z} \>  \\
    \end{dcases}
$
\end{center}

$\\*$which can be rewritten as
\begin{center}
$
HP_{n}^{(w)}(A) \cong  \hat{H}^{w-n}(C_{w},\bar{\mathfrak{m}}^{\otimes w})
$
\end{center}
$\\*$for the $C_{w}$ action on $\bar{\mathfrak{m}}^{\otimes w}$ explained above.

\end{corollary}
Since $\mathbb{Q}$ is a projective $\mathbb{Q}[C_{w}]$-module, applying Equations (\ref{tor1}), (\ref{tor2}), (\ref{ext1}), and (\ref{ext2}), we get 
\begin{corollary}\label{PQ}
Let $A = \mathbb{Q}[x_1,x_2,\hdots,x_d]/\mathfrak{m}^2$ where $\mathfrak{m}$ is the ideal $(x_1,x_2,\hdots,x_d)$.
\medskip
$\\*$ Then for $w = 0$

\begin{center}
$
HP_{n}^{(0)}(A) \cong 
    \begin{dcases}
        \mathbb{Q}  & n \textrm{ even}  \\
        0 & \textrm{else} \\
    \end{dcases}
$
\end{center}

\medskip

$\\*$ and for $w > 0$

\begin{center}
$
HP_{n}^{(w)}(A) \cong 0.
$
\end{center}

\end{corollary}
Applying Lemma \ref{1}, Lemma \ref{2}, Equation (\ref{rel1}), and Equation (\ref{rel2}) to Corollary \ref{HPNW3}, we get 

\begin{corollary}\label{Zperiodic}

Let $A = \mathbb{Z}[x_1,x_2,\hdots,x_d]/\mathfrak{m}^2$ where $\mathfrak{m}$ is the ideal $(x_1,x_2,\hdots,x_d)$. Then

\begin{center}
$
HP_{n}^{(0)}(A) \cong 
    \begin{dcases}
        \mathbb{Z} & n  \textrm{ even} \\
        0 & n \textrm{ odd} \\
    \end{dcases}
$
\end{center}

\medskip

$\\*$ and for $w > 0$

\begin{center}
$
HP_{n}^{(w)}(A) \cong 
    \begin{dcases} 
       \displaystyle\bigoplus_{\substack{m \mid w \\ m \equiv w \Mod{2}}}\displaystyle\bigoplus_{\omega_{m,d}} \mathbb{Z}/\big(\tfrac{w}{m}\big)
         & n = w + 2i \textrm{ for } i \in \mathbb{Z} \\
        \\
       \displaystyle\bigoplus_{\substack{m \mid w \\ m \not\equiv w \Mod{2}}}\displaystyle\bigoplus_{\omega_{m,d}} \mathbb{Z}/2 & n = w + 1 + 2i\textrm{ for } i \in \mathbb{Z} \\
    \end{dcases}
$
\end{center}
$\\*$where $\omega_{m,d} = \{\textrm{all cycle families of words length $m$ and cycle length $m$ in $x_1,\hdots,x_d$}\}$  has order \newline $\frac{1}{m}\sum_{i \vert m}\mu(m/i) d^{i}$.
\end{corollary}

\section{Acknowledgements}
 The author would like to thank her advisor Ayelet Lindenstrauss for bringing this problem to her attention and for the endless amount of useful conversations on the road to solving it. The author would also like to thank Michael Mandell for many helpful discussions and to thank Michael Larsen for showing her a better way to define $\omega_{m,d}$ using the M\"obius function. Some of this work was supported under the Air Force Office of Scientific Research Grant FA9550-16-1-0212.
\bibliography{Bibby2}

\begin{thebibliography}{Goo86}

\bibitem[Goo86]{GW}
Thomas~G. Goodwillie.
\newblock Relative algebraic {$K$}-theory and cyclic homology.
\newblock {\em Ann. of Math. (2)}, 124(2):347--402, 1986.

\bibitem[Lod98]{CH}
Jean-Louis Loday.
\newblock {\em Cyclic homology}, volume 301 of {\em Grundlehren der
  Mathematischen Wissenschaften [Fundamental Principles of Mathematical
  Sciences]}.
\newblock Springer-Verlag, Berlin, second edition, 1998.
\newblock Appendix E by Mar\'{i}a O. Ronco, Chapter 13 by the author in
  collaboration with Teimuraz Pirashvili.

\end{thebibliography}
\bibliographystyle{alpha}

\end{document}